\documentclass[psamsfonts]{amsart}
\usepackage{amssymb,amsfonts}
\usepackage{enumerate}
\usepackage{mathrsfs}
\usepackage{url}

\newtheorem{thm}{Theorem}[section]
\newtheorem{cor}[thm]{Corollary}
\newtheorem{prop}[thm]{Proposition}

\newtheorem{conj}[thm]{Conjecture}

\theoremstyle{definition}
\newtheorem{defn}[thm]{Definition}
\newtheorem{defns}[thm]{Definitions}

\newtheorem{notn}[thm]{Notation}

\theoremstyle{remark}
\newtheorem{rem}[thm]{Remark}

\makeatletter
\let\c@equation\c@thm
\makeatother
\numberwithin{equation}{section}

\DeclareMathOperator*{\type}{type}
\DeclareMathOperator*{\ord}{ord}

\newcommand{\cpx}[1]{\|#1\|}
\newcommand{\acl}{\ell}
\newcommand{\dft}{\delta}
\newcommand{\D}{\mathscr{D}}
\newcommand{\Dst}{\D_\st}
\newcommand{\clDx}{\overline{\D}}
\newcommand{\clD}{\overline{\Dst}}
\newcommand{\Da}[1]{\D^{#1}_{\mathrm{st}}}
\newcommand{\Dax}[1]{\D^{#1}}
\newcommand{\clDa}[1]{\overline{\Da{#1}}}
\newcommand{\clDax}[1]{\overline{\Dax{#1}}}
\newcommand{\st}{\mathrm{st}}
\newcommand{\xpdd}[1]{\hat{#1}}
\newcommand{\val}{\mathrm{val}}

\newcommand{\N}{{\mathbb N}}

\newcommand{\Z}{{\mathbb Z}}
\newcommand{\Q}{{\mathbb Q}}
\newcommand{\Nn}{\Z_{\ge0}}

\newcommand{\sS}{{\mathcal S}}

\title{Integer complexity: Stability and self-similarity}
\author{Harry Altman and Juan Arias de Reyna}
\date{October 16, 2025}

\begin{document}

\begin{abstract}
Define $\cpx{n}$ to be the \emph{complexity} of $n$, the smallest number of ones
needed to write $n$ using an arbitrary combination of addition and
multiplication (the smallest number of $1$'s in a $(1,+,\cdot)$-expression for
$n$).  The set $\D$ of \emph{defects}, differences $\dft(n):=\cpx{n}-3\log_3 n$,
is known to be a well-ordered subset of $[0,\infty)$, with order type
$\omega^\omega$.  This is proved by showing that, for any $s$, there is a finite
set $\sS_s$ of certain multilinear polynomials, called low-defect polynomials,
such that $\dft(n)\le s$ if and only if one can write $n =
f(3^{k_1},\ldots,3^{k_r})3^{k_{r+1}}$ for some $f\in\sS_s$. \cite{paperwo,
theory}

In this paper we show that, in addition to it being true that $\D$ (and thus
$\overline{\D}$) has order type $\omega^\omega$, this set satisifies a sort of
self-similarity property, with $\overline{\D}' = \overline{\D} + 1$.  This is
proven by restricting attention to \emph{substantial} low-defect polynomials,
ones that can be themselves written efficiently in a certain sense, and showing
that in a certain sense the values of these polynomials at powers of $3$ have
complexity equal to the na\"ive upper bound most of the time.

As an application, we also prove that, under appropriate conditions on $a$ and
$b$, numbers of the form $b(a3^k+1)3^\ell$ will, for all sufficiently large $k$,
have complexity equal to the na\"ive upper bound.  These results resolve various
earlier conjectures of the second author \cite{Arias}.
\end{abstract}

\maketitle

\section{Introduction}

The \emph{integer complexity} of a natural number $n$,
denoted $\cpx{n}$, is the least number of $1$'s needed to write it using any
combination of addition and multiplication, with the order of the operations
specified using  parentheses grouped in any legal nesting.  For instance, $n=11$
has a complexity of $8$, since it can be written using $8$ ones as \[
11=(1+1+1)(1+1+1)+1+1,\] but not with any fewer than $8$.  This notion was
introduced by Kurt Mahler and Jan Popken \cite{MP} and by Richard Guy \cite{Guy,
UPINT}.

Integer complexity is approximately logarithmic, and in particular it satisfies the bounds
\begin{equation*}\label{eq1}
3 \log_3 n= \frac{3}{\log 3} \log  n\le \cpx{n} \le \frac{3}{\log 2} \log n =
3\log_2 n
,\qquad n>1.
\end{equation*}
The lower bound can be deduced from the results of Mahler and Popken, and was
explicitly proved by John Selfridge \cite{Guy}. It is attained with equality for
$n=3^k$ for all $k \ge1$.
The upper bound by contrast is not sharp; see \cite{upbds} for more on that, and
see \cite{Amano} regarding better upper bounds that work only for most inputs.
See also \cite{He} regarding algorithms for computing integer complexity.

Based on the lower bound, earlier work \cite{paper1} introduced the
notion of the \emph{defect} of $n$, denoted $\dft(n)$, which is defined to be
the difference $\cpx{n}-3\log_3 n$.
\begin{defn}
We define $\D$ to be the set of all defects,
\[ \D = \{ \dft(n) = \cpx{n}-3\log_3 n: n \ge 1 \}.\]
\end{defn}
This set has some unexpected structure:

\begin{thm}[\cite{paperwo}]
\label{oldwo}
The set $\D$ of defects of all natural numbers is a well-ordered subset of the
real line, with order type $\omega^\omega$.  Moreover, for $k\ge 1$, the set
$\D\cap[0,k)$ has order type $\omega^k$.
\end{thm}

In this paper, we show that this set has additional structure beyond that, and
apply this to get lower bounds on the complexity of numbers of the form
$b(a3^k+1)3^\ell$.

\subsection{Self-similarity of the defect set}

As mentioned above, the set $\D$ of all defects is a well-ordered
subset of the real line, with order type $\omega^\omega$.
Moreover, it is known that the limit of the initial $\omega^k$ defects of $\D$
occurs at precisely $k$. \cite{paperwo}

In this paper we will frequently work with the closure $\clDx$ instead of
directly with the defect set $\D$.  This change of emphasis
simplifies many statements.  The set $\clDx$ also has a
structural characterization which will be given in Theorem~\ref{closure-intro}.

The main theorem of the paper may be stated in these terms:

\begin{thm}[Self-similarity theorem]
\label{selfsim-intro}
\[\clDx' = \clDx + 1\]
\end{thm}

Here $S'$ denotes the derived set (that is, the set of limit points) of $S$. Although $\clDx' = \D'$, we write $\clDx'$ so that we have the same set on both
sides of the equation.

This theorem tells us that the set $\clDx$ has a ``self-similarity''
property; if one shifts it over by $1$ one obtains its limit points.  So, the
original set $\clDx$ can be obtained by taking the set $\clDx+1$ and attaching a
``tail'' to the left of each point to make it a limit point (at least, if one
knows where to put the points in this ``tail'').

A striking illustration is provided by looking at the sets $\clDx \cap [k,
k+1]$; or better yet, if one translates them to obtain $(\clDx \cap [k,
k+1])-k$.  Then each set in the sequence looks like the previous, except that
each point has sprouted a tail and become a limit point.

Let note some other ways that this theorem may be understood.  We begin by
introducing some notation.

\begin{notn}
\label{index}
Given a well-ordered set $S$ and an ordinal $\alpha$, we will use $S[\alpha]$ to denote the $\alpha$'th
element of $S$; and when $\alpha>0$, will use $S(\alpha)$ to denote
$S[-1+\alpha]$.  (Here, for $\alpha>0$, $-1+\alpha$ denotes the unique $\beta$
such that $1+\beta=\alpha$.  So if $\alpha>0$ is finite, then
$-1+\alpha=\alpha-1$ and $S(\alpha)=S[\alpha-1]$, and if $\alpha$ is infinite,
then $-1+\alpha=\alpha$ and $S(\alpha)=S[\alpha]$.)
\end{notn}

The reason for introducing the $1$-indexed $S(\alpha)$ notation instead of
always using the $0$-indexed $S[\alpha]$ notation is that we want to think of
closed sets as being $1$-indexed.   (Here, the term ``indexing'' is being used
with its meaning from computer programming; to say we are $0$-indexing means
that the initial point gets an index of $0$, while to say we are $1$-indexing
means it gets an index of $1$.) This is because, for $S\subseteq\mathbb{R}$
well-ordered, we have \[\overline{S}(\alpha)=\sup_{\beta<\alpha}S[\beta]\]
whenever $\alpha$ is a limit ordinal, $\alpha$ is finite, or $S$ is discrete;
see Proposition~\ref{explainingclosures}.  (The set $\D$ is not discrete, but in
Section~\ref{stableintro} we will introduce a variant of it, $\Dst$, that we
will show is discrete (Corollary~\ref{closure-weak}), with $\clD=\clDx$
(Theorem~\ref{closure}); and much of the paper will be written in terms of
$\Dst$.)  Because of this, we will frequently write equations in terms of
$\clDx(\alpha)$ for convenience, but it should be kept in mind that this is
encoding a supremum over $\D$ or $\Dst$ itself; i.e., such statements are as
much about $\D$ and $\Dst$ as they are about $\clDx$.

With this notation, the second half of Theorem~\ref{oldwo} may be written
\begin{equation}
\label{omegak}
\clDx(\omega^k)= \sup_{\alpha<\omega^k}\D[\alpha] = k.
\end{equation}

Now, for another point of view on this equation, we might also write it as
\[\clDx(\omega^k \cdot 1) = k = 0 + k = \clDx(1) + k.\]

With this point of view, Equation~\eqref{omegak} may be seen as a special case
of Theorem~\ref{selfsim-intro}, which we may rephrase as follows:

\begin{thm}[Combined index-value shift relation]
\label{cj8weak-intro}
Given $1\le \alpha<\omega^\omega$ an ordinal and $k$ a whole number,
\[\clDx(\omega^k \alpha)=\clDx(\alpha)+k.\]
\end{thm}

This statement may be further strengthened.  This theorem, and the equivalent
Theorem~\ref{selfsim-intro}), discuss what happens for $\clDx$ as a whole.
In previous papers the set $\D$ was broken down into sets $\Dax{u}$; here,
$\Dax{u}$ is the set of defects of numbers $n$ with $\cpx{n}\equiv u\pmod{3}$.
(See \cite{paperwo} and \cite[Section~1.5]{intdft} for more about the reason
for this decomposition.)

These individual components are a little harder to work with than $\D$ as a
whole.  In \cite{intdft} it was proven that Equation~\eqref{omegak} extends
to them however:
\begin{equation}
\label{omegak3}
\clDax{u+k}(\omega^k) = \clDax{u}(1) + k.
\end{equation}

Here we go further, proving the analogue of Theorem~\ref{cj8weak-intro} for the
component sets, which appeared earlier as \cite[Conjecture~8]{Arias}:

\begin{thm}[Split index-value shift relation]
\label{cj8}
Given $1\le \alpha<\omega^\omega$ an ordinal, $k$ a whole number, and $u$ a
 congruence class modulo $3$,
\[\clDax{u+k}(\omega^k \alpha)=\clDax{u}(\alpha)+k.\]
\end{thm}

Note we have rephrased this conjectures somewhat from its original language; see
the Appendix of \cite{paperwo}, as well as \cite{compactum}, for more on
translating between this language and the original, and see
Theorem~\ref{cj8orig} later in the paper for a form that is closer to the
original.

Also, just as one can rewrite Theorem~\ref{cj8weak-intro} as
Theorem~\ref{selfsim-intro}, one may rewrite Theorem~\ref{cj8} similarly;
see Corollary~\ref{selfsim3}.

See Appendix~\ref{addchains} for a conjectured analogue of this theorem for
addition chains.

\subsubsection{Defects of expressions}
\label{dftexp}

Theorem~\ref{selfsim-intro} has an important consequence, which we will prove
later:
\begin{thm}[Defect closure theorem]
\label{closure-intro}
\[\clDx = \D + \Nn \]
\end{thm}

This result has an interesting interpretation.  Consider the following
definitions:
\begin{defns}
\label{defnexpr}
Let $E$ be an expression in $(1,+,\cdot)$.  Define $\val(E)$ to mean the
value of $E$, and $\cpx{E}$ to mean the number of $1$'s in $E$.  Then define the
defect of $E$, $\dft(E)$, by $\dft(E)=\cpx{E}-3\log_3 \val(E)$.
\end{defns}

Then if $E$ is any $(1,+,\cdot)$-expression, $\dft(E)=\dft(\val(E))+k$
for some nonnegative integer $k$.  Conversely, if we consider any defect
$\dft(n)$ and any nonnegative integer $k$, then $\dft(n)+k=\dft(E)$ for some
not-necessarily-minimal expression $E$ with $\val(E)=n$.

Thus we may rephrase Theorem~\ref{closure-intro}:

\begin{thm}[Expression defect theorem]
\label{dftexprthm-intro}
\[ \clDx = \{ \dft(E): E\textrm{ a $(1,+,\cdot)$-expression}\} \]
\end{thm}

That is, the closure of the set of all defects of numbers is equal to the set of
all defects of expressions.

There is also an analogue of this result when we restrict to particular
congruence classes modulo $3$, which appears later as
Corollary~\ref{dftexprsplit}.

\subsubsection{Stability and stable defects}
\label{stableintro}

Consider the following definition and theorem, which was also
\cite[Conjecture~1]{Arias}:

\begin{defn}
A number $n$ is called \emph{stable} if $\cpx{3^k n}=3k+\cpx{n}$ holds for every
$k \ge 0$.
\end{defn}

\begin{thm}[Stability theorem; {\cite[Theorem~13]{paper1}}]
\label{stab}
For any natural number $n$, there exists $K\ge 0$ such that $n3^K$ is stable;
i.e., for all $k\ge K$, one has
\[ \cpx{n3^k} = \cpx{n3^K} + 3(k-K). \]
\end{thm}

One might phrase this as, when $k$ gets large, eventually increasing $k$ by $1$
will always increase the complexity by $3$.  It is a stability theorem for
numbers of the form $a3^k$.

With this, we can define \emph{stable defects}, which allow us to give a
stronger statement of Theorem~\ref{closure-intro}.

\begin{defn}
Define a stable defect to be a defect of the form $\dft(n)$ for stable $n$.
We then define $\Dst$, the set of stable defects, to be the set of $\dft(n)$ for
all stable natural numbers $n$.
\end{defn}

See Section~\ref{dftsubsec} for other characterizations of stable defects.

Many of the theorems stated above are more naturally stated in terms of $\clD$
rather than $\clDx$, and will be proved in that form.  However, these sets are
the same, yielding a stronger version of Theorem~\ref{closure-intro}:

\begin{thm}[Stable defect closure theorem]
\label{closure}
We have:
\begin{enumerate}
\item $\clDx = \clD$
\item $\clD = \Dst + \Nn$
\item $\Dst + \Nn = \D + \Nn$
\end{enumerate}
\end{thm}

We will prove this in Section~\ref{corsec}.  Note that since $\Dst\subseteq\D$
and (by an application of Theorem~\ref{stab}) $\D \subseteq \Dst + \Nn$, the
hard part of this is proving part (2).

While we have thus far discussed results in terms of defects $\D$ and
$\clDx$, the rest of this paper will be primarily written in
terms of stable defects $\Dst$ and $\clD$.

\subsection{The method: Substantial low-defect polynomials}
\label{substintro}

In order to prove Theorem~\ref{selfsim-intro}, we introduce a
notion we call substantial low-defect polynomials.

A low-defect polynomial is a particular type of multilinear polynomial,
introduced in \cite{paperwo} and expanded upon in \cite{theory}, used for
studying numbers with defect below a given bound; see Section~\ref{dftsec} for
details.  In \cite{paperwo} it was proved that, given any positive real number
$s$, one can write down a finite set of low-defect polynomials $\sS$ such that
every number $n$ with $\dft(n)\le s$ can be written in the form
$f(3^{n_1},\ldots,3^{n_d})3^{n_{d+1}}$ for some $f\in\sS$; and that, moreover,
such an $n$ can always be represented ``efficiently'' in such a fashion.
Moreover, it was shown in \cite{theory} that one can choose $\sS$ such that for
any $f\in\sS$, one has $\deg f\le s$.  (Note that the degree of a low-defect
polynomial is always equal to the number of variables it uses, since low-defect
polynomials are multilinear and always include a term containing all the
variables.)

The defects arising from a low-defect polynomial $f$ are bounded above by a
quantity we denote $\dft(f)$ (see Definition~\ref{deltaf}).  In \cite{theory} it
was shown that this quantity satisifes the inequality
\begin{equation}
\label{dftineq}
\dft(f) \ge \dft(a) + \deg f \ge \dft_{\st}(a) + \deg f,
\end{equation}
where $a$ is the leading coefficient of $f$.  Here $\dft_{\st}$ represents the
``stable defect'', which is defined in Definition~\ref{stabdef}.

\begin{defn}
We define a \emph{substantial low-defect polynomial}
to be a polynomial that saturates \eqref{dftineq}, one where
\[ \dft(f) = \dft_{\st}(a) + \deg f. \]
\end{defn}

We call such polynomials ``substantial'' because, as we will show in
Proposition~\ref{cj9prop}, one has that $f$ is substantial if and only if the
degree of $f$ is maximal among all low-defect polynomials $g$ with
$\dft(g)=\dft(f)$.  Lower-degree polynomials will have their contributions to
the clustering of defects below $\dft(f)$ absorbed and overshadowed by those of
larger degree, rendering them ``insubstantial''.

However, even a substantial polynomial $f$ will only affect the clustering of
defects below $\dft(f)$ if its defects do indeed approach $\dft(f)$, rather than
capping out at some smaller $\dft(f)-k$.  We will show though as
Proposition~\ref{usu} that the former case always occurs; if $f$ is substantial,
then numbers coming from it will ``usually'' have the expected complexity and
defect, and the set of exceptions is ``small'' in an appropriate sense.

Note that we only prove that the set of exceptions is small in a fairly weak
sense; however, we will prove in a future paper \cite{hyperplanes} that the
exceptional set in fact small in a stronger sense.

\subsection{Applications to earlier conjectures and variants}

We will apply the main theorem of this paper to a series of conjectures by the
second author (as well as some variants). \cite{Arias}  Actually, although here
we treat our first application, Theorem~\ref{cj2thm}, as an application of
Theorem~\ref{selfsim-intro}, it is actually possible to go the other way and use
Theorem~\ref{cj2thm} to prove Theorem~\ref{selfsim-intro}; there is an
equivalence between the two.  (Strictly speaking, this requires ignoring the
computability requirement.)  However, we will not demonstrate this equivalence
in this paper.

\subsubsection{Application: Stability of $a(b3^k+1)$}
\label{cj2-intro}

We discussed in Section~\ref{stableintro} the phenomenon of stabilization for
numbers of the form $a3^k$.  We will apply the main theorem of this paper to
prove that a similar phenomenon occurs for numbers of the form $b(a3^k+1)$:

\begin{thm}[Degree-$1$ stability theorem]
\label{cj2thm}
We have:
\begin{enumerate}
\item Suppose $a$ is stable.  Then there exists $K$ such that, for all $k\ge K$
and all $\ell\ge 0$,
\[\cpx{(a3^k+1)3^\ell}=\cpx{a}+3k+3\ell+1 .\]
\item Suppose $ab$ is stable and $\cpx{ab}=\cpx{a}+\cpx{b}$.  Then there exists
$K$ such that, for all $k\ge K$ and all $\ell\ge 0$,
\[\cpx{b(a3^k+1)3^\ell}=\cpx{a}+\cpx{b}+3k+3\ell+1 .\]
\end{enumerate}
Moreover, in both these cases, it is possible to algorithmically compute how
large $K$ needs to be.
\end{thm}

This theorem essentially resolves a conjecture of the second author
\cite[Conjecture~2]{Arias}, who suggested that given any two natural numbers $a$
and $b$, there exists $K$ such that for all $k\ge K$,
\begin{equation}
\label{cj2eqn}
\cpx{b(a3^k+1)} = \cpx{a} + \cpx{b} + 3k + 1.
\end{equation}

This conjecture is too strong as stated, as there are cases where there turn out
to be simpler ways of writing the numbers in question.  For instance, consider
the case of $b=2$, $a=1094$.  Since
$\cpx{1094}=22$, if \cite[Conjecture~2]{Arias} were true as stated, then for $k$
sufficiently large, one would have
\[ \cpx{2(1094\cdot3^k+1)}=25 + 3k.\]  However,
it turns out that $\cpx{2\cdot 1094}=\cpx{2188}=22$ as well (since $2188 = 3^7
+1$), which means that, for any $k\ge 0$,
\[ \cpx{2(1094\cdot3^k+1)} = \cpx{2188\cdot3^k + 2} \le 24 + 3k. \]

However, Theorem~\ref{cj2thm} shows that Equation~\eqref{cj2eqn} is true if we
require the additional hypothesis that $\cpx{ab}=\cpx{a}+\cpx{b}$, i.e., that
there is no more efficient way of writing $ab$ than factoring it into $a$ and
$b$, and additionally require that $ab$ is stable.  The obvious modification for
the case $b=1$ is also true.

\subsubsection{The off-by-one case}
\label{secdragon}

We can also get an analogue of Theorem~\ref{cj2thm} that holds even if the
complexities are off by one, if the polynomial considered is just barely
insubstantial (see Section~\ref{secsubst}).

\begin{thm}[Off-by-one stability theorem]
\label{dragons}
We have:
\begin{enumerate}
\item Suppose $ab$ is stable, $\cpx{a}+\cpx{b}=\cpx{ab}+1$, and $b>1$.
Suppose further that $a$ is stable.  Then there exists $K$ such that for all
$k\ge K$,
\[\cpx{b(a3^k+1)}=\cpx{a}+\cpx{b}+3k+1 .\]
\item Supose further that $b$ is also stable.  Then there exists $K$ such that
for all $k\ge K$ and $\ell\ge 0$,
\[\cpx{b(a3^k+1)3^\ell}=\cpx{a}+\cpx{b}+3k+3\ell+1 .\]
\end{enumerate}
Moreover, in both these cases, it is possible to algorithmically compute how
large $K$ needs to be.
\end{thm}

Note that one cannot extend this theorem to cases that are off by $2$, as
demonstrated by the case of $2(1094\cdot3^k+1)$ considered above.  So, there is
a certain irregularity to this case.  We will address the case of more general
degree-$1$ low-defect polynomials in a future paper \cite{deg1}.

\subsubsection{Applications to earlier conjectures}
\label{sectable}
The paper
\cite{Arias} of the second author made
eleven conjectures about integer complexity.  Some of these conjectures
(Conjecture~1 and Conjectures~3--7) have since then been proven or had salvages
proven \cite{paper1, paperwo}.

This leaves Conjecture~2, Conjecture~8, and Conjectures~9--11.  We discussed
Conjecture~2 in Section~\ref{cj2-intro}; a salvaged version appears here as
Theorem~\ref{cj2thm}, which we will prove in Section~\ref{deg1}.
Theorem~\ref{dragons} is also related.  There are other ways to modify
\cite[Conjecture~2]{Arias} by adding additional hypotheses; we hope to prove
other variants in a future paper \cite{deg1}.

Conjecture~8, as previously mentioned, is a rephrasing of Theorem~\ref{cj8} that
we will prove in Section~\ref{secproof}; we will also discuss Conjecture~8
further in Section~\ref{cj9sec}.

Finally, as for Conjectures~9--11, those will not be discussed in detail in this
introduction, but salvages of them are proved in Section~\ref{cj9sec}.  Thus we
put to rest all remaining conjectures from \cite{Arias}.

\subsection{Structure of the paper}

Section~\ref{dftsec} reviews preliminaries from previous papers on the
integer complexity defect.  Also included are general preliminaries from
topology and order theory.  Section~\ref{secsubst} introduces substantial
polynomials and explains how they
work, and then Section~\ref{secproof} uses them to prove Theorem~\ref{cj8} and
related statements, including a weak version of Theorem~\ref{cj2thm}, and
\cite[Conjectures~9--11]{Arias}.  Finally
Section~\ref{deg1} focuses on the degree $1$ case; it proves
Theorem~\ref{cj2thm}, Theorem~\ref{dragons}, and corollaries of these.

\section{Preliminaries}
\label{dftsec}

In this section we will review existing results on defects and low-defect
polynomials, as well some other preliminary results we will need.

\subsection{The defect, stable defect, and stable complexity}
\label{dftsubsec}

We start with some basic facts about the defect:

\begin{prop}[{\cite[Theorem~2.1]{theory}}]
\label{oldprops}
We have:
\begin{enumerate}
\item For all $n$, $\dft(n)\ge 0$.
\item For $k\ge 0$, $\dft(3^k n)\le \dft(n)$, with equality if and only if
$\cpx{3^k n}=3k+\cpx{n}$.  The difference $\dft(n)-\dft(3^k n)$ is a nonnegative
integer.
\item A number $n$ is stable if and only if for any $k\ge 0$, $\dft(3^k
n)=\dft(n)$.
\item If the difference $\dft(n)-\dft(m)$ is rational, then $n=m3^k$ for some
integer $k$ (and so $\dft(n)-\dft(m)\in\mathbb{Z}$).
\item Given any $n$, there exists $k$ such that $3^k n$ is stable.
\item For a given defect $\alpha$, the set $\{m: \dft(m)=\alpha \}$ has either
the form $\{n3^k : 0\le k\le L\}$ for some $n$ and $L$, or the form $\{n3^k :
0\le k\}$ for some $n$.  The latter occurs if and only if $\alpha$ is the
smallest defect among $\dft(3^k n)$ for $k\in \mathbb{Z}$.
\item If $\dft(n)=\dft(m)$, then $\cpx{n}=\cpx{m} \pmod{3}$.
\item $\dft(1)=1$, and for $k\ge 1$, $\dft(3^k)=0$.  No other integers occur as
$\dft(n)$ for any $n$.
\item If $\dft(n)=\dft(m)$ and $n$ is stable, then so is $m$.
\end{enumerate}
\end{prop}

We will want to consider the set of all defects:

\begin{defn}
We define the \emph{defect set} $\mathscr{D}$ to be $\{\dft(n):n\in\N\}$, the
set of all defects.
In addition, for $u$ a congruence class modulo $3$, we
define \[\mathscr{D}^u = \{\dft(n):n>1,\ \cpx{n}\equiv u\pmod{3}\}.\]
\end{defn}

\begin{prop}
\label{disjbasic}
For distinct congruence classes $u$ modulo $3$, the sets $\mathscr{D}^u$ are
disjoint.
\end{prop}

\begin{proof}
This follows from part (7) of Proposition~\ref{oldprops}.
\end{proof}

The paper \cite{paperwo} also defined the notion of a \emph{stable defect}:

\begin{defn}
\label{stabdef}
We define a \emph{stable defect} to be the defect of a stable number, and define
$\mathscr{D}_\st$ to be the set of all stable defects.  Also, for $a$ a
congruence class modulo $3$, we define $\mathscr{D}^u_\st=\mathscr{D}^u \cap
\mathscr{D}_\st$.
\end{defn}

Because of part (9) of Theorem~\ref{oldprops}, this definition makes sense; a
stable defect $\alpha$ is not just one that is the defect of some stable number,
but one for which any $n$ with $\dft(n)=\alpha$ is stable.  Stable defects can
also be characterized by the following proposition from \cite{paperwo}:

\begin{prop}[{\cite[Proposition~2.4]{theory}}]
\label{modz1}
A defect $\alpha$ is stable if and only if it is the smallest
$\beta\in\mathscr{D}$ such that $\beta\equiv\alpha\pmod{1}$.
\end{prop}

We can also define the \emph{stable defect} of a given number, which we denote
$\dft_\st(n)$.

\begin{defn}
For a positive integer $n$, define the \emph{stable defect} of $n$, denoted
$\dft_\st(n)$, to be $\dft(3^k n)$ for any $k$ such that $3^k n$ is stable.
(This is well-defined as if $3^k n$ and $3^\ell n$ are stable, then $k\ge \ell$
implies $\dft(3^k n)=\dft(3^\ell n)$, and $\ell\ge k$ implies this as well.)
\end{defn}

Note that the statement ``$\alpha$ is a stable defect'', which earlier we were
thinking of as ``$\alpha=\dft(n)$ for some stable $n$'', can also be read as the
equivalent statement ``$\alpha=\dft_\st(n)$ for some $n$''.

Similarly we have the stable complexity:
\begin{defn}
For a positive integer $n$, define the \emph{stable complexity} of $n$, denoted
$\cpx{n}_\st$, to be $\cpx{3^k n}-3k$ for any $k$ such that $3^k n$ is stable.
\end{defn}

We then have the following facts relating the notions of $\cpx{n}$, $\dft(n)$,
$\cpx{n}_\st$, and $\dft_\st(n)$:

\begin{prop}
\label{stoldprops}
We have:
\begin{enumerate}
\item $\dft_\st(n)= \min_{k\ge 0} \dft(3^k n)$.
\item $\dft_\st(n)$ is the smallest $\alpha\in\mathscr{D}$ such that
$\alpha\equiv \dft(n) \pmod{1}$.  In particular, if two stable defects are
congruent modulo $1$, then they are equal.
\item $\cpx{n}_\st = \min_{k\ge 0} (\cpx{3^k n}-3k)$.
\item $\dft_\st(n)=\cpx{n}_\st-3\log_3 n$.
\item $\dft_\st(n) \le \dft(n)$, with equality if and only if $n$ is stable.
\item $\cpx{n}_\st \le \cpx{n}$, with equality if and only if $n$ is stable.
\item $\cpx{3n}_\st = \cpx{n}_\st+3$.
\item If $\dft_\st(n)=\dft_\st(m)$, then $\cpx{n}_\st\equiv\cpx{m}_\st\pmod{3}$.
\item $\cpx{nm}_\st \le \cpx{n}_\st + \cpx{m}_\st$.
\end{enumerate}
\end{prop}

\begin{proof}
Parts (1)-(8) are Proposition~2.7 from \cite{intdft}.  Part (9) is Proposition~9
from Section~7 of \cite{paperalg}.
\end{proof}

Remember that $1$ is not stable, so one has $\cpx{1}=1$ and $\cpx{1}_\st=0$.

Note, by the way, that just as $\mathscr{D}_\st$ can be characterized either as
defects $\dft(n)$ with $n$ stable or as defects $\dft_\st(n)$ for any $n$,
$\mathscr{D}^u_\st$ can be characterized either as defects $\dft(n)$ with $n$
stable and $\cpx{n}\equiv u\pmod{3}$, or as defects $\dft_\st(n)$ for any $n$
with $\cpx{n}_\st\equiv u\pmod{3}$.

We also make the following definition:
\begin{defn}
For $n\in N$, define $\Delta(n) = \dft(n)-\dft_\st(n) = \cpx{n}-\cpx{n}_\st$.
\end{defn}

By the above, one always has $\Delta(n)\in\Nn$.

Also, in order to further discuss stabilization, it is useful here to define:
\begin{defn}
\label{defk}
Given $n\in \N$, define $K(n)$ to be the smallest $k$ such that $n3^k$ is
stable.
\end{defn}

Then it was shown in \cite{paperalg} that:
\begin{thm}
\label{computk}
The function $K$ is computable; the function $n\mapsto\cpx{n}_\st$ is
computable; and the set of stable numbers is a computable set.
\end{thm}

We will use this later in proving that the bounds in Theorem~\ref{cj2thm} and
its variants in Section~\ref{deg1} can be computed.

There is one more fact about stability that we will use repeatedly.

\begin{prop}[{\cite[Section~7, Corollary~1]{paperalg}}]
We have:
\label{goodfac}
\begin{enumerate}
\item If $N$ is stable, $N=n_1\cdots n_k$, and
$\cpx{N}=\cpx{n_1}+\ldots+\cpx{n_k}$, then the $n_i$ are also stable.
\item If $n_i$ are stable numbers, $N=n_1\cdots n_k$, and
$\cpx{N}_\st=\cpx{n_1}_\st+\ldots+\cpx{n_k}_\st$, then $N$ is also stable.
\end{enumerate}
\end{prop}

Finally, let us formally note how defects of expressions relate to defects of
numbers.

\begin{prop}
\label{dftexprprop}
Let $E$ be a $(1,+,\cdot)$-expression.  Then $\dft(E)=\dft(\val(E))+k$ where
$k=\cpx{E}-\cpx{\val(E)}$ is a nonnegative integer.  Conversely, if we consider
any defect $\dft(n)$ and any nonnegative integer $k$, then $\dft(n)+k=\dft(E)$
for some $(1,+,\cdot)$-expression $E$ with $\val(E)=n$ and $\cpx{E}-\cpx{n}=k$.
\end{prop}

\begin{proof}
Given $E$, $\dft(E)=\cpx{E}-3\log_3(\val(E))$.  Since $E$ is an expression for
$\val(E)$, $\cpx{E}\ge \cpx{\val(E)}$; let $k=\cpx{E}-\cpx{\val(E)}$.  Then
$\dft(E)=\dft(\val(E))+k$.

Conversely, given $n$ and $k$, let $E'$ be a minimal $(1,+,\cdot)$-expression
for $n$, so $\val(E')=n$ and $\cpx{E'}=\cpx{n}$.  Then let $E$ be the product of
$E'$ with $k$ additional factors of $1$.  Then $val(E)=n$ and
$\cpx{E}=\cpx{n}+k$, so $\dft(E)=\dft(n)+k$.
\end{proof}

\subsection{Low-defect polynomials and the exceptional set}
\label{secpoly}

We represent the set of numbers with defect at most $r$ by substituting in
powers of $3$ into certain multilinear polynomials we call \emph{low-defect
polynomials}.  We will associate with each one a ``base complexity'' to form a
\emph{low-defect pair}.  These notions can also be formalized in terms of
\emph{low-defect expression} or \emph{low-defect trees}, which we will discuss
shortly.

\begin{defn}
\label{polydef}
We define the set $\mathscr{P}$ of \emph{low-defect pairs} as the smallest
subset of $\Z[x_1,x_2,\ldots]\times \N$ such that:
\begin{enumerate}
\item For any constant polynomial $k\in \N\subseteq\Z[x_1, x_2, \ldots]$ and any
$C\ge \cpx{k}$, we have $(k,C)\in \mathscr{P}$.
\item Given $(f_1,C_1)$ and $(f_2,C_2)$ in $\mathscr{P}$, we have $(f_1\otimes
f_2,C_1+C_2)\in\mathscr{P}$, where, if $f_1$ is in $d_1$ variables and $f_2$ is
in $d_2$ variables,
\[ (f_1\otimes f_2)(x_1,\ldots,x_{d_1+d_2}) :=
	f_1(x_1,\ldots,x_{d_1})f_2(x_{d_1+1},\ldots,x_{d_1+d_2}). \]
\item Given $(f,C)\in\mathscr{P}$, $c\in \N$, and $D\ge \cpx{c}$, we have
$(f\otimes x_1 + c,C+D)\in\mathscr{P}$ where $\otimes$ is as above.
\end{enumerate}

The polynomials obtained this way will be referred to as \emph{low-defect
polynomials}.  If $(f,C)$ is a low-defect pair, $C$ will be called its
\emph{base complexity}.  If $f$ is a low-defect polynomial, we will define its
\emph{absolute base complexity}, denoted $\cpx{f}$, to be the smallest $C$ such
that $(f,C)$ is a low-defect pair.
We will also associate to a low-defect polynomial $f$ the \emph{augmented
low-defect polynomial}
\[ \xpdd{f} = f\otimes x_1; \]
if $f$ is in $d$ variables, this is $fx_{d+1}$.
\end{defn}

So, for instance, $(3x_1+1)x_2+1$ is a low-defect polynomial, as is
$(3x_1+1)(3x_2+1)$, as is $(3x_1+1)(3x_2+1)x_3+1$, as is
\[2((73(3x_1+1)x_2+6)(2x_3+1)x_4+1).\]  In this paper we will primarily concern
ourselves with low-defect pairs $(f,C)$ where $C=\cpx{f}$, so in much of what
follow, we will dispense with the formalism of low-defect pairs and just discuss
low-defect polynomials.

Note that the degree of a low-defect polynomial is also equal to the number of
variables it uses; see Proposition~\ref{polystruct}.
Also note that augmented low-defect polynomials are never themselves low-defect
polynomials; as we will see in a moment (Proposition~\ref{polystruct}),
low-defect polynomials always have nonzero constant term, whereas augmented
low-defect polynomials always have zero constant term.  We can also observe
that low-defect polynomials are in fact read-once polynomials as discussed in
for instance \cite{ROF}.

Note that we do not really care about what variables a low-defect polynomial is
in -- if we permute the variables of a low-defect polynomial or replace them
with others, we will still regard the result as a low-defect polynomial.  From
this perspective, the meaning of $f\otimes g$ could be simply regarded as
``relabel the variables of $f$ and $g$ so that they do not share any, then
multiply $f$ and $g$''.  Helpfully, the $\otimes$ operator is associative not
only with this more abstract way of thinking about it, but also in the concrete
way it was defined above.

From \cite{paperwo}, we have the following propositions about low-defect
polynomials:

\begin{prop}[{\cite[Proposition~4.2]{paperwo}}]
\label{polystruct}
Suppose $f$ is a low-defect polynomial of degree $d$.  Then $f$ is a
polynomial in the variables $x_1,\ldots,x_d$, and it is a multilinear
polynomial, i.e., it has degree $1$ in each of its variables.  The coefficients
are non-negative integers.  The constant term is nonzero, and so is the
coefficient of $x_1\cdots x_d$, which we will call the \emph{leading
coefficient} of $f$.
\end{prop}

\begin{prop}[{\cite[Proposition~2.10]{theory}}]
\label{basicub}
If $f$ is a low-defect polynomial of degree $d$, then
\[\cpx{f(3^{n_1},\ldots,3^{n_d})}\le \cpx{f}+3(n_1+\ldots+n_d).\]
and
\[\cpx{\xpdd{f}(3^{n_1},\ldots,3^{n_{d+1}})}\le \cpx{f}+3(n_1+\ldots+n_{d+1}).\]
\end{prop}

Because of this, it makes sense to define:

\begin{defn}
Given a low-defect pair $(f,C)$ (say of degree $r$) and a number $N$, we will
say that $(f,C)$ \emph{efficiently $3$-represents} $N$ if there exist
nonnegative integers $n_1,\ldots,n_r$ such that
\[N=f(3^{n_1},\ldots,3^{n_r})\ \textrm{and}\ \cpx{N}=C+3(n_1+\ldots+n_r).\]
We will say $(\xpdd{f},C)$ efficiently
$3$-represents $N$ if there exist $n_1,\ldots,n_{r+1}$ such that
\[N=\xpdd{f}(3^{n_1},\ldots,3^{n_{r+1}})\ \textrm{and}\ 
\cpx{N}=C+3(n_1+\ldots+n_{r+1}).\]
More generally, we will also say $f$ $3$-represents $N$ if there exist
nonnegative integers $n_1,\ldots,n_r$ such that $N=f(3^{n_1},\ldots,3^{n_r})$,
and similarly with $\xpdd{f}$.
\end{defn}

Note that if $(f,C)$ (or $(\xpdd{f},C)$) efficiently $3$-represents some $N$,
then $(f,\cpx{f})$ (respectively, $(\xpdd{f},\cpx{f})$ efficiently
$3$-represents $N$, which means that in order for $(f,C)$ (or $(\xpdd{f},C)$ to
$3$-represent anything efficiently at all, we must have $C=\cpx{f}$.  However it
is still worth using low-defect pairs rather than just low-defect polynomials
since we may not always know $\cpx{f}$.  In our applications here, where we wish
to perform computations by means of these objects, taking the time to compute
$\cpx{f}$, rather than just making do with an upper bound, may not be desirable.

For this reason it makes sense to use ``$f$ efficiently $3$-represents $N$'' to
mean ``some $(f,C)$ efficiently $3$-represents $N$'' or equivalently
``$(f,\cpx{f})$ efficiently $3$-reperesents $N$''.  Similarly with $\xpdd{f}$.

In keeping with the name, numbers $3$-represented by low-defect polynomials, or
their augmented versions, have bounded defect.  Let us make some definitions
first:

\begin{defn}
\label{deltaf}
Given a low-defect pair $(f,C)$, we define $\dft(f,C)$, the defect of $(f,C)$,
to be $C-3\log_3 a$, where $a$ is the leading coefficient of $f$.  We will also
define $\dft(f)$ to mean $\dft(f,\cpx{f})$, since much of the time we will not
be concerned with keeping track of base complexities.
\end{defn}

One thing worth noting about defects of polynomials, that has not been noted
previously:

\begin{prop}
\label{fdft}
Let $(f,C)$ and $(g,D)$ be low-defect pairs.  If $\dft(f,C)=\dft(g,D)$, then
$C\equiv D \pmod{3}$.
\end{prop}

\begin{proof}
The proof is exactly the same as the proof of part (7) from
Proposition~\ref{oldprops}.  If $\dft(f,C)=\dft(g,D)$, then
\[ C - 3\log_3 a = D - 3\log_3 b, \]
where $a$ and $b$ are the leading coefficients of $f$ and $g$, respectively.
So $C-D = 3\log_3(\frac{a}{b})\in \Z$, and so in particular it is a rational
number, meaning $\log_3(\frac{a}{b})$ is in turn a rational number, which can
only happen if it is in fact an integer, so $3\mid C-D$ as required.
\end{proof}

\begin{defn}
\label{fdftdef}
Given a low-defect pair $(f,C)$ of degree $r$, we define
\[\dft_{f,C}(n_1,\ldots,n_r) =
C+3(n_1+\ldots+n_r)-3\log_3 f(3^{n_1},\ldots,3^{n_r}).\]
We will also define $\dft_f$ to mean $\dft_{f,\cpx{f}}$ when we are not
concerned with keeping track of base complexities.
\end{defn}

Then we have:

\begin{prop}[{\cite[Proposition~2.15]{intdft}}]
\label{dftbd}
Let $(f,C)$ be a low-defect pair of degree $r$, and let $n_1,\ldots,n_{r+1}$ be
nonnegative integers.
\begin{enumerate}
\item We have
\[ \dft(\xpdd{f}(3^{n_1},\ldots,3^{n_{r+1}}))\le \dft_{f,C}(n_1,\ldots,n_r)\]
and the difference is an integer.
\item We have \[\dft_{f,C}(n_1,\ldots,n_r)\le\dft(f,C)\]
and if $r\ge 1$, this inequality is strict.
\item The function $\dft_f$ is strictly increasing in each variable, and
\[ \dft(f) = \sup_{n_1,\ldots,n_d} \dft_f(n_1,\ldots,n_d).\]
\end{enumerate}
\end{prop}

The defects we get from a low-defect polynomial $f$ form a well-ordered set of
order type approximately $\omega^d$:

\begin{prop}[{\cite[Proposition~2.16]{intdft}}]
\label{indivtype}
Let $f$ be a low-defect polynomial of degree $d$.  Then:
\begin{enumerate}
\item The image of $\dft_f$ is a well-ordered subset of $\mathbb{R}$, with
order type $\omega^d$.
\item The set of $\dft(N)$ for all $N$ $3$-represented by the augmented
low-defect polynomial $\xpdd{f}$ is a well-ordered subset of $\mathbb{R}$, with
order type at least $\omega^d$ and at most $\omega^d(\lfloor \delta(f)
\rfloor+1)<\omega^{d+1}$.  The same is true if $f$ is used instead of the
augmented version $\xpdd{f}$.
\end{enumerate}
\end{prop}

The reason we care about low-defect polynomials is that all numbers of
sufficiently low defect can be efficiently $3$-represented by them.  First,
some definitions:

\begin{defn}
A natural number $n$ is called a \emph{leader} if it is the smallest number with
a given defect.  By part (6) of Proposition~\ref{oldprops}, this is equivalent
to saying that either $3\nmid n$, or, if $3\mid n$, then $\dft(n)<\dft(n/3)$,
i.e., $\cpx{n}<3+\cpx{n/3}$.
\end{defn}

\begin{defn}
For any real $s\ge0$, define the set $\overline{A}_s$ to be 
\[\overline{A}_s := \{n\in\mathbb{N}:\dft(n)\le s\}.\]
Define the set $\overline{B}_s$ to be 
\[
\overline{B}_s:= \{n \in \overline{A}_s :~~n~~\mbox{is a leader}\}.
\]
\end{defn}

The use of the overline here is to contrast $\overline{A}_s$ and
$\overline{B}_s$, which use nonstrict inequalities in the definition, with the
earlier $A_s$ and $B_s$ from \cite{paper1}, which were the same but using strict
inequalities.

Obviously, one has:

\begin{prop}
\label{arbr}
For every $n\in \overline{A}_r$, there exists a unique $m\in \overline{B}_r$ and
$k\ge 0$ such that $n=3^k m$ and $\dft(n)=\dft(m)$; then $\cpx{n}=\cpx{m}+3k$.
\end{prop}

\begin{proof}
This is exactly Proposition~2.6 from \cite{paperwo}, except that we are looking
at $n$ with $\dft(n)\le r$, instead of $\dft(n)<r$.
\end{proof}

Then, it was shown in \cite{theory} (Theorem~A.7) that:

\begin{thm}
\label{covering}
For any real $s\ge 0$, there exists a finite set $\sS_s$ of low-defect pairs
satisfying the following conditions:
\begin{enumerate}
\item For any $n\in \overline{B}_s$, there is some low-defect pair in $\sS_s$
that efficiently $3$-represents $n$.
\item Each pair $(f,C)\in \sS_s$ satisfies $\dft(f,C)\le s$, and hence $\deg
f\le \lfloor s\rfloor$.
\end{enumerate}
We refer to such a set $\sS_s$ as a \emph{good covering} of $\overline{B}_s$.
\end{thm}

Moreover, it was shown in \cite{paperalg} (see Algorithm~5 and the appendix)
that:
\begin{thm}
\label{goodcomput}
Given a real number $s$ of the form $q+r\log_3 n$ with $n\in \N$, $q,r\in \Q$,
it is possible to algorithmically compute a good covering of $\overline{B}_s$.
\end{thm}

(The assumption on the form of $s$ here is inessential and is just to restrict
to a computable subset of real numbers; one can state the theorem more generally
than this.)

In this paper we will be concerned with proving that certain low-defect pairs,
the \emph{substantial} low-defect pairs, efficiently $3$-represent ``most'' of
the numbers that they $3$-represent.  In order to discuss this, it helps to make
the following definition:

\begin{defn}
\label{except}
Let $(f,C)$ be a low-defect pair of degree $d$. Define its
\emph{exceptional set} to be
\[
\{(n_1,\ldots,n_d): \cpx{f(3^{n_1},\ldots,3^{n_d})}_\st<C+3(n_1+\ldots+n_d)\}
\]
We will also say ``the exceptional set of $f$'' to simply mean the exceptional
set of $(f,\cpx{f})$.
\end{defn}

Finally, one more property of low-defect polynomials we will need is the
following:

\begin{prop}[{\cite[Proposition~3.24]{theory}}]
\label{ineq}
Let $f$ be a low-defect polynomial, and suppose that $a$ is the leading
coefficient of $f$.  Then $\cpx{f}\ge \cpx{a} + \deg f$, which also implies
$\cpx{f}\ge \cpx{a}_\st + \deg f$.

In particular, $\dft(f) \ge \dft(a) + \deg f$ and
$\dft(f) \ge \dft_\st(a) + \deg f$.
\end{prop}

With this, we have the preliminary notions and terminology out of the way.
However, we will also take a moment to discuss some alternate formalisms.

\subsection{Low-defect expressions and trees}

Now, it is mathematically convenient to phrase things in terms of polynomials,
but sometimes we want a finer-grained view of things.  Rather than look at a
polynomial $f$, we may want to look at the expression that gives rise to it.

That is to say, if we have a low-defect polynomial $f$, it was constructed
according to rules (1)--(3) in Definition~\ref{polydef}; each of these rules
though gives a way not just of building up a polynomial, but an expression.  For
instance, we can build up the polynomial $4x+2$ by using rule (1) to make $2$,
then using rule (3) to make $2x+1$, then using rule (2) to make $2(2x+1)=4x+2$.
The polynomial $4x+2$ itself does not remember its history, of course; but
perhaps we want to remember its history -- in which we do not want to consider
the \emph{polynomial} $4x+2$, but rather the \emph{expression} $2(2x+1)$, which
is different from the expression $4x+2$.

So, with that, we define:

\begin{defn}
A \emph{low defect expression} is defined to be an expression in positive
integer constants, $+$, $\cdot$, and some number of variables, constructed
according to the following rules:
\begin{enumerate}
\item Any positive integer constant $c$ by itself forms a low-defect expression.
\item Given two low-defect expressions using disjoint sets of variables, their
product is a low-defect expression.  If $E_1$ and $E_2$ are low-defect
expressions, we will use $E_1 \otimes E_2$ to denote the low-defect expression
obtained by first relabeling their variables to be disjoint and then multiplying
them.
\item Given a low-defect expression $E$, a positive integer constant $c$, and a
variable $x$ not used in $E$, the expression $E\cdot x+c$ is a low-defect
expression.  (We can write $E\otimes x+c$ if we do not know in advance that $x$
is not used in $E$.)
\end{enumerate}
\end{defn}

And, naturally, we also define:

\begin{defn}
We define an \emph{augmented low-defect expression} to be an expression of the
form $E\cdot x$, where $E$ is a low-defect expression and $x$ is a variable not
appearing in $E$.  If $E$ is a low-defect expression, we also denote the
augmented low-defect expression $E\otimes x$ by $\xpdd{E}$.
\end{defn}

It is clear from the definitions that evaluating a low-defect expression yields
a low-defect polynomial, and that evaluating an augmented low-defect expression
yields an augmented low-defect polynomial.  Note also that low-defect
expressions are read-once expressions, so, as mentioned earlier, low-defect
polynomials are read-once polynomials.

All of the results of Section~\ref{secpoly}, which were stated in terms of
low-defect pairs, can instead be stated in terms of low-defect expressions,
though we will not restated them in this way here.  Note that for this we need
the notion of the complexity of a low-defect expression:

\begin{defn}
We define the complexity of a low-defect expression $E$, denoted $\cpx{E}$, as
follows:
\begin{enumerate}
\item If $E$ is a positive integer constant $n$, we define $\cpx{E}=\cpx{n}$.
\item If $E$ is of the form $E_1 \cdot E_2$, where $E_1$ and $E_2$ are
low-defect expressions, we define $\cpx{E}=\cpx{E_1}+\cpx{E_2}$.
\item If $E$ is of the form $E' \cdot x + c$, where $E'$ is a low-defect
expression, $x$ is a variable, and $c$ is a positive integer constant, we define
$\cpx{E}=\cpx{E'}+\cpx{c}$.
\end{enumerate}
\end{defn}

In addition, we can helpfully represent a low-defect expression by a rooted
tree, with the vertices and edges both labeled by positive integers.  Some
information is lost in this representation, but nothing of much relevance.  This
representation does away with such problems as, for instance, $4$ and $2\cdot 2$
being separate expressions.  In addition, trees can be treated more easily
combinatorially, which will prove useful in a sequel paper \cite{seqest}.

\begin{defn}
Given a low-defect expression $E$, we define a corresponding \emph{low-defect
tree} $T$, which is a rooted tree where both edges and vertices are labeled with
positive integers.  We build this tree as follows:
\begin{enumerate}
\item If $E$ is a constant $n$, $T$ consists of a single vertex labeled with
$n$.
\item If $E=E'\cdot x + c$, with $T'$ the tree for $E'$, $T$ consists of $T'$
with a new root attached to the root of $T'$.  The new root is labeled with a
$1$, and the new edge is labeled with $c$.
\item If $E=E_1 \cdot E_2$, with $T_1$ and $T_2$ the trees for $E_1$ and $E_2$
respectively, we construct $E$ by ``merging'' the roots of $E_1$ and $E_2$ --
that is to say, we remove the roots of $E_1$ and $E_2$ and add a new root, with
edges to all the vertices adjacent to either of the old roots; the new edge
labels are equal to the old edge labels.  The label of the new root is equal
to the product of the labels of the old roots.
\end{enumerate}
\end{defn}

We can define an associated base complexity for these too:

\begin{defn}
The complexity of a low-defect tree, $\cpx{T}$, is defined to be the smallest
$\cpx{E}$ among all low-defect expressions yielding $T$.
\end{defn}

We also, for expressions and trees, have the following concrete expression for
the complexity:

\begin{prop}[{\cite[Proposition~3.23]{theory}}]
\label{treecpx}
We have:
\begin{enumerate}
\item Let $E$ be a low-defect expression.  Then $\cpx{E}$ is equal to the sum of
the complexities of all the integer constants occurring in $E$.
\item Let $T$ be a low-defect tree.  Then
\[ \cpx{T} = \sum_{e\ \textrm{an edge}} \cpx{w(e)} +
	\sum_{v\ \textrm{a leaf}} \cpx{w(v)}
	+ \sum_{\substack{v\ \textrm{a non-leaf vertex}\\ w(v)>1}} \cpx{w(v)},\]
where $w$ denotes the label of the given vertex or edge.
\end{enumerate}
\end{prop}

Also worth noting is the following:

\begin{prop}
\label{treelead}
Let $T$ be a low-defect tree, let $V$ and $E$ be its vertex set and edge set,
let $f$ be the low-defect polynomial arising from it, and let $N$ be its leading
coefficient.  Then $N$ is equal to the product of all vertex labels in $T$, and
$\deg f = |V|-1 = |E|$.
\end{prop}

\begin{proof}
This is just Proposition~3.14 from \cite{theory} combined with
Proposition~\ref{polystruct} above.
\end{proof}

Note that for a low-defect polynomial $f$, $\cpx{f}$ can be equivalently
characterized as the smallest $\cpx{T}$ among all expressions $E$ or trees $T$
yielding $f$.

So, we get a chain from more information preserved to least information
preserved.  Most specific is the low-defect expression $E$; this can then be
represented by a tree $T$; this can then be evaluated to get a polynomial $f$,
which we can associate with a base complexity $\cpx{T}$ to get the low-defect
pair $(f,\cpx{T})$; and finally we can just look at $f$ itself, getting the
low-defect pair $(f,\cpx{f})$.

In truth, we could add a few more steps here, such as a tree-pair $(T,C)$ or
expression-pair $(E,C)$; or, most specific of all, a low-defect expression $E$
where each numerical constant $n$ is replaced by a specific
$(1,+,\cdot)$-expression that represents it.  But none of this will be necessary
here; expressions and trees will suffice for now.

\subsection{Some topological and order preliminaries}

Before we proceed, we should make notes of some facts from topology and order
theory that we will need.

\begin{prop}[\cite{MOtopologies}]
\label{topologies}
Let $X$ be a totally ordered set with the least upper bound property, and let
$S\subseteq X$ be closed.  Then the subspace topology and the order topology on
$S$ coincide.
\end{prop}

This proposition allows us to ignore questions of what topology we are using.
One key reason we need the above proposition is to get the following:

\begin{prop}
\label{limitpoints}
Let $X$ be a totally ordered set with the least upper bound property, and let
$S\subseteq X$ be a closed, well-ordered set.  Then the limit points of $S$
within $X$ are the points of $S$ of the form $S(\omega\alpha)$ for ordinals
$\alpha>0$ (see Notation~\ref{index}).
\end{prop}

\begin{proof}
The set $\{S(\omega\alpha): \omega\le\omega\alpha<\type(S)\}$ is the set of
limit points of $S$ within itself under its order topology, and by
Proposition~\ref{topologies}, this coincides with the subspace topology it
inherits as a subset of $X$.  Since $S$ is closed in $X$, this is the same as
the set of limit points of $S$ within $X$.
\end{proof}

We will need to know a few more things about well-orders, closures, and limit
points.

First off, we need to know how indices in the closure relate to indices in the
original set.  We will focus on the case where the original set is discrete, as
will be the case for the sets we consider.  We will make use of the following
fact, a proof of which can be found in \cite{adcwo}, where it is
Proposition~5.5:

\begin{prop}
\label{closuretype}
Let $X$ be a totally ordered set, and let $S$ be a well-ordered subset of order
type $\alpha$.  Then $\overline{S}$ is also well-ordered, and has order type
either $\alpha$ or $\alpha+1$.  If $\alpha=\gamma+k$ where $\gamma$ is a limit
ordinal and $k$ is finite, then $\overline{S}$ has order type $\alpha+1$ if and
only if the initial segment of $S$ of order type $\gamma$ has a supremum in $X$
which is not in $S$.
\end{prop}

With this we observe:

\begin{prop}
\label{barshift}
Let $X$ be a totally-ordered set with the least upper bound property, and let
$S\subseteq X$ be well-ordered.  Suppose $\alpha$ is an ordinal
less than the order type of $S$, writing $\alpha=\gamma+k$ with $\gamma$ not a
sucessor and $k$ finite.  Let $T=\{S[\beta]:\beta<\gamma\}$.

Then if $\alpha$ is finite or if $\sup T\notin S$,
\[ S[\alpha] = \overline{S}(\alpha+1). \]
In particular, this always holds if $S$ is discrete.
If these conditions do not hold, then instead $S[\alpha]=\overline{S}(\alpha)$.
\end{prop}

\begin{proof}
Let $\eta = S[\alpha]$, and look at $S_\eta = \{ \zeta\in S: \zeta<\eta \}$ and
at $\overline{S}_\eta = \{ \zeta\in\overline{S}: \zeta<\eta \}$; note that
since subspace topologies commute with closures, $\overline{S}_\eta$ is in fact
the closure of $S_\eta$.  Now, certainly $S_\eta$ has order type $\alpha$.  So
by Proposition~\ref{closuretype}, $\overline{S_\eta}$ has order type either
$\alpha$ or $\alpha+1$, based on whether or not $T$ has a supremum in $S$.

If $\alpha$ is infinite, then $T$ is bounded above by $\eta$, so it has a
supremum in $X$ (since $X$ has the least upper bound property), and this
supremum lies outside $S$ by assumption.  So $\overline{S}_\eta$ has order type
$\alpha+1$, which since $\alpha$ is infinite is equal to $-1+(\alpha+1)$; that
is, $S[\alpha]=\overline{S}(\alpha+1)$.

On the other hand, if $\alpha$ is finite, then $S_\eta$ and $\overline{S}_\eta$
are certainly equal, so both have order type $\alpha$.  Since $\alpha$ is
finite, $\alpha+1=1+\alpha$, and so $S[\alpha]=\overline{S}(\alpha+1)$.

Finally, if the conditions do not hold, that is if is $\alpha$ is infinite and
$\sup T\in S$, then $\overline{S}$ has order type $\alpha$, which is $\alpha$ is
infinite is equal to $-1+\alpha$, and therefore
$S[\alpha]=\overline{S}(\alpha)$.
\end{proof}

More generally, we can relate indices in the closure to indices in the
original set:

\begin{prop}
\label{explainingclosures}
Let $X$ be a totally-ordered set with the least upper bound property, and let
$S\subseteq X$ be well-ordered.  Suppose $\alpha>0$ is an ordinal less than the
order type of $S$; we may also allow $\alpha$ equal to the order type of $S$ if
$S$ is bounded above in $X$.  Write $\alpha=\gamma+k$ with $\gamma$ not a
sucessor and $k$ finite.  Let $T=\{S[\beta]:\beta<\gamma\}$.
Then if $\alpha$ is a limit ordinal, or $\alpha$ is finite, or $\sup T\notin S$,
\[ \overline{S}(\alpha) = \sup_{\beta<\alpha} S[\beta]. \]
In particular, this always holds if $S$ is discrete or if $\alpha$ is a power of
$\omega$.
\end{prop}

\begin{proof}
If $\alpha$ is a successor, then $\sup_{\beta<\alpha} S[\beta]=S[\alpha-1]$, and
by Proposition~\ref{barshift}, $S[\alpha-1]=\overline{S}(\alpha)$.  (Note that
in invoking Proposition~\ref{barshift}, we have used the hypothesis that either
$\alpha$ is finite or $\sup T\notin S$.)

Contrariwise, if $\alpha$ is a limit, then since $\overline{S}$ is closed,
certainly we have $\overline{S}(\alpha) = \sup_{0<\beta<\alpha}
\overline{S}(\beta)$ (since the latter point is in $\overline{S}$ and must
follow directly after all $\overline{S}(\beta)$ for $0<\beta<\alpha$).  Now,
given $\beta<\alpha$, by Proposition~\ref{barshift},
$S[\beta]\in\{\overline{S}(\beta), \overline{S}(\beta+1)\}$.  From this we
can conclude that for any $\delta<\alpha$, we have
$\sup_{k\in\omega}S[\delta+k] = \sup_{k\in\omega}\overline{S}(\delta+k)$.
Since the blocks of $\omega$ have the same limits, the overall limits are also
the same, and therefore
\[ \overline{S}(\alpha) = \sup_{0<\beta<\alpha} \overline{S}(\beta) =
\sup_{\beta<\alpha} S[\beta], \]
as desired.
\end{proof}

We will additionally need to know how the types of limits at a point can be
related to its index.

\begin{notn}
Suppose $S$ is a well-ordered set of order type $\alpha>0$, and with the Cantor
normal form of $\alpha$ being $\alpha=\omega^{\alpha_0}a_0 + \ldots +
\omega^{\alpha_r}a_r$.  Then we will define $\ord \alpha = \alpha_r$.
\end{notn}

It is a well-known fact from order theory that $\omega^{\ord \alpha}$ is the
smallest order type among nonzero final segments of $\alpha$.  We will require
another characterization:

\begin{prop}
\label{ordercof}
Suppose of $S$ is a well-ordered set of order type $\alpha>0$.  Then
$\ord\alpha$ is equal to the largest $\beta$ such that $S$ has a cofinal subset
of order type $\omega^\beta$.
\end{prop}

While this fact is also well known, we did not find a good reference for it so
we supply a proof of our own.

\begin{proof}
Firstly, $S$ has a cofinal subset of order type $\omega^{\ord \alpha}$, because
the final Cantor block is such a set.  It thus only remains to show that nothing
larger is possible.

Suppose $T$ is cofinal in $S$, and that the order type of $T$ is equal to
$\omega^\beta$.  Let $F$ be the final Cantor block of $S$.  Then the order type
of $T\cap F$ is at most $\omega^{\ord\alpha}$.  But since $F$ is a final segment
of $S$, $T\cap F$ is a final segment of $T$; and since $T$ is cofinal in $S$, it
is a nonempty final segment.  Since the order type of $T$ is a power of
$\omega$, all nonempty final segments of $T$, including $T\cap F$, have that
same order type, $\omega^\beta$.  Therefore, $\beta\le\ord\alpha$.
\end{proof}

The main reason we care about this is the following:

\begin{prop}
\label{order}
Let $X$ be a totally-ordered set with the least upper bound property, and let
$S\subseteq X$ be well-ordered.  Consider a point $\eta\in\overline{S}$, and
write $\eta=\overline{S}(\alpha)$, so $\alpha>0$.

Then $\ord\alpha$ is the largest $\beta$ such that there is a subset of $S$
of order type $\omega^\beta$ with supremum equal to $\eta$.  In particular,
$\ord\alpha = 0$ if and only if $\eta$ is an isolated point of $S$, and
$\ord\alpha > 0 $ if and only if $\eta$ is a limit point of $S$.
\end{prop}

\begin{proof}
First, if $T$ is a subset of $S$ with $\sup T=\eta$ and with order type equal to
$\omega^\beta$, then $T$ is also a subset of $\overline{S}$, so
$\beta\le\ord\alpha$ by Proposition~\ref{ordercof}.

Now, for the reverse, let $U$ be the final Cantor block of
$\overline{S}\cap(-\infty,\eta]$, and let $T=U\cap S$, so $U=\overline{T}$.  So
$U$ has order type $\omega^{\ord\alpha}$.  Thus, by
Proposition~\ref{closuretype}, $T$ must also have order type
$\omega^{\ord\alpha}$ (including if $\ord\alpha=0$, since if $T$ has order type
$0$ rather than $1$, so would $U$).

Also, $\eta\in U=\overline{T}$, so we must have $\sup T=\eta$.  So $T$ is a
subset of $S$ of order type $\omega^{\ord\alpha}$ with $\sup T=\eta$,
proving the claim.
\end{proof}

Finally, one last key fact that we will use about well-orders is the following:

\begin{prop}
\label{cutandpaste}
We have:
\begin{enumerate}
\item If $S$ is a well-ordered set and $S=S_1\cup\ldots\cup S_n$, and $S_1$
through $S_n$ all have order type less than $\omega^k$, then so does $S$.
\item If $S$ is a well-ordered set of order type $\omega^k$ and
$S=S_1\cup\ldots\cup S_n$, then at least one of $S_1$ through $S_n$ also has
order type $\omega^k$.
\end{enumerate}
\end{prop}

Proofs of the more general statements this is a special case of can be found in
\cite{carruth} and \cite{wpo}; for a proof of precisely this statement, it is
proved from these more powerful principles as Proposition~5.4 in \cite{paperwo}.

\section{Substantial polynomials}
\label{secsubst}

The key new concept in this paper is that of the substantial low-defect
polynomial.

\begin{defn}
Let $(f,C)$ be a low-defect pair, and let $a$ be the leading coefficient of $f$.
We will say $(f,C)$ is \emph{substantial} if $C = \cpx{a}_\st + \deg f$.  We
will say $f$ is substantial if $\cpx{f} = \cpx{a}_\st + \deg f$.  (Since for a
low-defect pair to be substantial we must have $C=\cpx{f}$ by
Proposition~\ref{ineq}, we will often just talk about substantial
polynomials and ignore the formalism of pairs.)
\end{defn}

We call such polynomials (or pairs) ``substantial'' because, among all
low-defect pairs $(f,C)$ with a fixed value of $\dft(f,C)$, these are the ones
of maximum degree (see also Section~\ref{substintro}).

\begin{prop}
\label{polydft}
Suppose $(f,C)$ is a low-defect pair; then we may write $\dft(f,C)=\eta+k$,
where $\eta$ is a stable defect and $k$ is a whole number.
\end{prop}

\begin{proof}
Since $\dft(f,C) = C-3\log_3 a$ where $C\ge\cpx{a}$ (by Proposition~\ref{ineq}),
we have
\[ \dft(f,C) = (C-\cpx{a}) + \dft(a) = (C-\cpx{a}) + \Delta(a) + \dft_\st(a). \]
\end{proof}

\begin{prop}
\label{cj9prop}
Suppose $(f,C)$ is a low-defect pair, and write $\dft(f,C)=\dft(q)+k$, where $q$
is stable and $k$ is a whole number.  Let $a$ be the leading coefficient of $f$.
Then
\[ \deg f + (C-\deg f-\cpx{a}_\st) = k .\]
In particular, we always have $\deg f\le k$, and $(f,C)$ is substantial if and
only if $\deg f=k$.
\end{prop}

\begin{proof}
Since $C-3\log_3 a = \dft(q)+k$, we have that $\dft(a) \equiv \dft(q) \pmod{1}$.
So, by part (2) of Proposition~\ref{stoldprops}, since $q$ is stable, we have
$\dft_\st(a)=\dft(q)$.

Therefore,
\[C-\cpx{a}_\st = C - 3\log_3 a -\dft_\st(a)= k,\]
as required.  The second part is then just Proposition~\ref{ineq}
and the definition of ``substantial''.
\end{proof}

The question then is, how can we determine whether a given low-defect pair is
substantial?  One easy criterion is the following:

\begin{prop}
\label{basecase}
If $\dft(f,C)<\deg f+1$, then $(f,C)$ is substantial.
\end{prop}

\begin{proof}
If $\dft(f,C)<\deg f+1$, then (letting $a$ be the leading coefficient of $f$)
we have
\[ C-\deg f-\cpx{a}_\st = \dft(f,C) -\dft_\st(a)-\deg f
< 1-\dft_\st(a)\le 1,\] so $C-\deg f -\cpx{a}_\st =0$, that is, $(f,C)$ is
substantial.
\end{proof}

However, substantial polynomials can be much more general than this.  The
easiest way to tell if a low-defect polynomial is substantial is if we know how
it was formed.

\begin{prop}
\label{substrecur}
Let $(f,C)$ be a low-defect pair.  Then:
\begin{enumerate}
\item If $f$ is a constant $n$ and $C=\cpx{n}$, then $(f,C)$ is substantial if
and only if $n$ is stable.
\item If $f = g_1 \otimes g_2$ and $C=D_1+D_2$, and $a$ is the leading
coefficient of $f$ and $b_i$ is the leading coefficient of $g_i$, then $(f,C)$
is substantial if and only if $(g_1, D_1)$ is substantial, $(g_2, D_2)$ is
substantial, $a$ is stable, and $\cpx{a}=\cpx{b_1}+\cpx{b_2}$.
\item If $f = g \otimes x + c$, and $C=D+\cpx{c}$, then $(f,C)$ is substantial
if and only if $(g,D)$ is substantial and $c=1$.
\end{enumerate}
\end{prop}

It is worth remembering here (per Theorem~\ref{computk}) that it can be computed
algorithmically whether a given number is stable or not.

\begin{proof}
For part (1), the leading coefficient of $f$ is $n$, and $\deg f=0$, so $(f,C)$
is substantial if and only if $C=\cpx{n}_\st+\deg f=\cpx{n}_\st$; since
$C=\cpx{n}$, this holds if and only if $n$ is stable.

For part (2), $(f, C)$ is substantial if and only if
\[
D_1 + D_2 = \cpx{b_1 b_2}_\st + \deg g_1 + \deg g_2.
\]
By part (9) of Proposition~\ref{stoldprops}, we know that $\cpx{b_1 b_2}_\st \le
\cpx{b_1}_\st + \cpx{b_2}_\st$, so in particular for $(f,C)$ to be substantial
we must have
\[
D_1 + D_2 \le \cpx{b_1}_\st + \cpx{b_2}_\st + \deg g_1 + \deg g_2.
\]
But we know by Proposition~\ref{ineq} that $D_i \ge \cpx{b_i}_\st + \deg g_i$,
so the only way this can happen is if $D_i = \cpx{b_i}_\st + \deg g_i$, and both
sides of the equation are in fact equal.  So we conclude that each $(g_i, D_i)$
is substantial, and moreover that $\cpx{a}_\st = \cpx{b_1}_\st + \cpx{b_2}_\st$.
Since $(f,C)$, $(g_1,D_1)$, and $(g_2, D_2)$ are all substantial, we know that
$a$, $b_1$, and $b_2$ are all stable, and so we conclude that $\cpx{a} =
\cpx{b_1} + \cpx{b_2}$.

This shows that if $(f,C)$ is substantial, then the conditions of part (2) are
satisfied; and checking the converse is straightforward.

For part (3), let $a$ be the leading coefficient of $g$, which is also the
leading coefficient of $f$; then $(f, C)$ is substantial if and only if
\[
D + \cpx{c} = \cpx{a}_\st + \deg g + 1.
\]
We know that $D\ge \cpx{a}_\st + \deg g$, and that $\cpx{c}=1$, so this holds if
and only if $D = \cpx{a}_\st + \deg g$ (i.e., $(g,D)$ is substantial) and
$\cpx{c}=1$ (i.e., $c=1$).
\end{proof}

So, for instance, we can conclude:

\begin{prop}
\label{exists}
If $\eta$ is a stable defect (say $\eta=\dft(q)$ for $q$ stable) and $k$ a whole
number, then there is a substantial polynomial $f$ with leading coefficient $q$
such that $\dft(f)=\eta+k$ and $\deg f=k$.
\end{prop}

\begin{proof}
Since $q$ is stable, by repeatedly applying Proposition~\ref{substrecur}, the
polynomial $(((qx_1+1)x_2+1)\cdots)x_k+1$ is substantial.
\end{proof}

It makes sense to define substantiality for low-defect trees and expressions,
too:

\begin{defn}
If $E$ is a low-defect expression, then we define $E$ to be substantial if
$(f,\cpx{E})$ is substantial, where $f$ is the low-defect polynomial obtained by
evaluating $E$.  Similarly, if $T$ is a low-defect tree, we define $T$ to be
substantial if $(f,\cpx{T})$ is substantial, where $f$ is the low-defect
polynomial arising from $T$.
\end{defn}

We will not actually require these latter notions for our proof, but they are
useful to keep in mind in order to get a picture of what substantial polynomials
look like and how we can form them.  Obviously the analogues of
Proposition~\ref{substrecur} will work just as well for expressions or trees; we
will not repeat the proof here.  Note that for trees, we can to some extent read
substantiality directly off the tree:

\begin{prop}
\label{substree}
Let $T$ be a low-defect tree with vertex set $V$, and let $N$ be the product of
the vertex labels.  Then $T$ is substantial if and only if all edge labels are
$1$, no leaf vertex labels are equal to $1$, $N$ is stable, and \[ \cpx{N} =
\sum_{\substack{v\in V \\ w(v)> 1}} \cpx{w(v)}, \]
where $w(v)$ denotes the label of $v$.
\end{prop}

\begin{proof}
By Proposition~\ref{treelead}, the leading coefficent of the polynomial
corresponding to $T$ is equal to $N$, and its degree is equal to the number of
edges $|E|$
(i.e., one less than the number of vertices).  So $T$ is substantial if and only
if $\cpx{T}=\cpx{N}_\st + |E|$.  Apply Proposition~\ref{treecpx}, and note
firstly
that $|E|\le\sum \cpx{w(e)}$, with equality if and only if $\cpx{w(e)}=1$,
i.e.~$w(e)=1$, for all edges $E$;
and secondly that since \[ N= \prod_{\substack{v\in V \\ w(v)> 1}} w(v), \]
we also have that $\cpx{N}$ is at most the rest of the sum, with equality if
and only if
 \[ \cpx{N} = \sum_{\substack{v\in V \\ w(v)> 1}} \cpx{w(v)}; \]
and none of the leaf labels are $1$; and of course $\cpx{N}_\st\le\cpx{N}$ with
equality if and only if $N$ is stable.
\end{proof}

Again, this is useful for getting a concrete idea of what substantial
polynomials look like.  We will be particularly concerned with the degree-$1$
case:

\begin{prop}
\label{1subst}
A low-defect polynomial of degree $1$ is substantial if and only if it can be
written as $ax+1$, for $a$ stable, or $b(ax+1)$, where $ab$ is stable and
$\cpx{ab}=\cpx{a}+\cpx{b}$.
\end{prop}

\begin{proof}
This is immediate from Proposition~\ref{substrecur} or
Proposition~\ref{substree}.
\end{proof}

\section{Proving the main theorems}
\label{secproof}

We now begin to prove the main theorem.  We start by proving a restricted
version of Theorem~\ref{closure}, where we show that $\clD \subseteq \D+\Nn$,
together with its analogue for when we split into congruence classes.

\begin{prop}
\label{subclosure}
We have:
\begin{enumerate}
\item $\clD \subseteq \Dst + \Nn$.
\item Say $q$ is a stable number, $k\in \Nn$, and let $\eta=\delta(q)+k$.  If
$\eta\in\clDa{u}$, then $u \equiv \cpx{q} + k \pmod{3}$.
\end{enumerate}
\end{prop}

\begin{proof}
Suppose $\eta\in \clD$.  Now, we know that the set $\Dst$ is well-ordered, so
this means that $\eta\in \overline{\Dst\cap[0,\eta]}$.  That is to say,
\[\eta = \sup (\Dst\cap[0,\eta]).\]

Choose a good covering
$\sS$ of $\overline{B}_\eta$ (see Theorem~\ref{covering}).  Now, we know that
\[\Dst\cap[0,\eta]\subseteq
\bigcup_{(f,C)\in\sS}
\{\dft(f(3^{k_1},\ldots,3^{k_d})): k_i\in\Nn,\ d=\deg f\}.
\]
So, $\sup(\Dst\cap[0,\eta])$ is at most the maximum of the suprema of these
individual sets.  But, applying Proposition~\ref{dftbd}, this means that
\[\eta = \sup (\Dst\cap[0,\eta]) \le \max_{(f,C)\in\sS} \dft(f,C)\le \eta,\]
where the last inequality comes from the definition of $\sS$ (see
Theorem~\ref{covering}, part (2)).

Therefore,
\[\eta = \max_{(f,C)\in\sS} \dft(f,C);\]
or, in other words, there is some $(f,C)\in\sS$ such that $\dft(f,C)=\eta$.

So, if $b$ is the leading coefficient of this $f$, then by
Proposition~\ref{ineq}, we have $C=\cpx{b}+\ell$ for some $\ell\in\Nn$, and so
\[ \eta = C - 3\log_3 b = \ell + \dft(b) = \dft_\st(b) + k, \]
where $k=\ell+\Delta(b)$.  This proves part (1).

For the second part, define
\[\zeta = \max (\{ \dft(f,C): (f,C)\in \sS,\ \dft(f,C)<\eta \} \cup
\{0\}).\]
We will show that if $\dft(n)\in(\zeta,\eta]$, then
$\cpx{n}\equiv\cpx{q}+k\pmod{3}$.  Since $\D$ is well-ordered, there is also
some interval $(\eta,\theta)$ that is free of defects not equal to $\eta$, so
this will show that $\eta\notin\clDa{u}$ for $u\not\equiv\cpx{q}+k\pmod{3}$.

So, suppose $\dft(n)\in(\zeta,\eta]$.  Then $n$ is efficiently $3$-represented
by some $(f,C)\in\sS$; and since $\zeta<\dft(n)\le\dft(f,C)$, this means we must
have $\dft(f,C)=\eta$.  Note that since $n$ is efficiently $3$-represented by
$(f,C)$, this means we have $\cpx{n}\equiv C\pmod{3}$.

For comparison, let us consider the polynomial
\[ g(x_1,\ldots,x_k) = (((qx_1+1)x_2)\ldots)x_k+1 \]
and the low-defect pair $(g,\cpx{q}+k)$; note that $\dft(g,\cpx{q}+k)=\eta$.
Since $\dft(f,C)=\dft(g,\cpx{q}+k)$, we conclude by
Proposition~\ref{fdft} that $C\equiv\cpx{q}+k\pmod{3}$.

So, if $\dft(n)\in(\zeta,\eta]$, then
\[ \cpx{n}\equiv C\equiv\cpx{q}+k\pmod{3}. \]
As noted above, this proves the second part of the theorem.
\end{proof}

Now we ask, given a point in $\clD\subseteq \D+\Nn$, what can we say about the
type of limit leading up to it?  That is to say, the points leading up to it
will form an $\omega^m$ for some $m$; the question is, what is $m$?  We will
answer this question shortly, but right now can only put an upper bound on it.

\begin{prop}
\label{step3}
Let $q$ be a stable number, let $\eta=\delta(q)+k$ for some $k\in \Nn$, and
suppose $\eta = \clD(\alpha)$.  Then $\ord \alpha\le k$.  Moreover, if
$\eta=\clDa{\cpx{q}+k}(\beta)$, then $\ord\beta\le k$.
\end{prop}

\begin{proof}
We prove both parts simultaneously, using Proposition~\ref{order}.

Let $\sS$ be a good covering of $\overline{B}_\eta$, and let
\[\zeta = \max (\{ \dft(f,C): (f,C)\in \sS,\ \dft(f,C)<\eta \} \cup
\{0\}).\]

Let us examine $\D\cap(\zeta,\eta]$.  If $\dft(n)\in(\zeta,\eta]$, then $n$ is
efficiently $3$-represented by $(\hat{f},C)$ for some $(f,C)\in \sS$.  Since
$\dft(n)>\zeta$, we must have $\dft(f,C)\ge \dft(n)>\zeta$ and so
$\dft(f,C)=\eta$.  Also, since the $3$-representation is efficient, we have that
$\dft(n)$ is a value of $\dft_{f,C}$.

Therefore, $\D\cap(\zeta,\eta]$ is contained in the union of the images of
finitely many $\dft_{f,C}$, with each $(f,C)$ having $\dft(f,C)=\eta$.  But by
Proposition~\ref{cj9prop}, this means that we have $\deg f\le k$ for each of
these, and so, by Proposition~\ref{indivtype} and \ref{cutandpaste}, the order
type of $\D\cap(\zeta,\eta]$ is strictly less than $\omega^{k+1}$.

So, if $S$ is a subset of $\Dst$ or $\Da{u}$ with supremum equal to $\eta$ and
order type $\omega^\gamma$, we must have $\gamma\le k$, as otherwise
$S\cap(\zeta,\eta]$ would have this same order type, but the latter's order type
must be less than $\omega^{k+1}$.  Therefore, by Proposition~\ref{order}, we
have $\ord\alpha\le k$ and $\ord\beta\le k$.
\end{proof}

We now prove a crucial fact about substantial polynomials: They usually give the
right complexity (and their outputs are usually stable).  That is, if $f$ is a
low-defect polynomial, then its exceptional set (Defintion~\ref{except}) is
small; most inputs do not lie in the exceptional set.

However, this proposition will only prove that the exceptional set is small in a
weak sense.  When $\deg f=1$, we conclude that the exceptional set is finite,
which is about as strong a conclusion as one could expect, and we will elaborate
more on this case in Section~\ref{deg1}.  But when $\deg f>1$, the resulting
conclusion is quite weak.  Fortunately, it is possible to prove that the
exceptional set is small in a much stronger sense, and we will do this in a
subsequent paper \cite{hyperplanes}.  For now, though, this weak notion of a
small exceptional set will be all that we prove and all that we need.

\begin{prop}
\label{usu}
Suppose $f$ is a substantial low-defect polynomial of degree $k$, and let $S$ be
its exceptional set.  Then the order type of $\dft_f(S)$ is less than
$\omega^k$.  Equivalently, the order type of the set
\[ E := \{ \dft_\st(f(3^{n_1},\ldots,3^{n_k})) : (n_1,\ldots,n_k)\in S\} \]
is less than $\omega^k$, as is the order type of the set
\[ \{ \dft(f(3^{n_1},\ldots,3^{n_k})) : (n_1,\ldots,n_k)\in S\}. \]
\end{prop}

\begin{proof}
The order type of $\dft_f(\Nn^k)$ is at most $\omega^k$ by
Proposition~\ref{indivtype}.  We want to show that the order type of $\dft_f(S)$
is strictly less than this, so assume the contrary, that they are equal.

Now, for any
$(n_1,\ldots,n_k)\in S$, we have
\[ \dft_\st(f(3^{n_1},\ldots,3^{n_k})) = \dft_f(n_1,\ldots,n_k) - \ell \]
for some whole number $\ell$ with $1\le \ell \le \lfloor \dft(f)\rfloor$.  (Here
we know $\ell \ge 1$ by the assumption that $(n_1,\ldots,n_k)\in S$, and we know
$\ell\le\lfloor\dft(f)\rfloor$ by Proposition~\ref{dftbd}.)
So define
\[ S_\ell := 
\{ (n_1,\ldots,n_k)\in S : 
\dft_\st(f(3^{n_1},\ldots,3^{n_k})) = \dft_f(n_1,\ldots,n_k)-\ell
\}.
\]
Then if we define $E_\ell = \dft_f(S_\ell)-\ell$, we can write
\[ E = \bigcup_{1\le\ell\le\lfloor\dft(f)\rfloor} E_\ell .\]

By our assumption that $\dft_f(S)$ has order type $\omega^k$, the same must be
true for at least one $E_\ell$ by Proposition~\ref{cutandpaste}.  However, since
$E_\ell+\ell$ and $\dft_f(S)$ both have order type $\omega^k$, the former must
be cofinal in the latter; so $\sup E_\ell = \dft(f) - \ell$, and therefore
$\dft(f)-\ell\in\clD$.

Suppose that $f$ has leading coefficient $q$, so $\dft(f)=\dft(q)+k$ with
$q$ stable.  Then $\dft(f)-\ell = \dft(q) + k - \ell$; since $\ell\ge 1$, this
means that $\dft(f) - \ell = \dft(q) + r$ for some integer $r$ strictly less
than $k$.

But by Propositions~\ref{stoldprops}, \ref{step3}, and \ref{subclosure}, this is
impossible.  If $r<0$, then by Proposition~\ref{subclosure} and part (2) of
Proposition~\ref{stoldprops}, the point $\dft(q)+r$ cannot lie in $\clD$ at all.
While if $r\ge 0$, if we write $\dft(q)+r = \clD(\alpha)$, then by
Proposition~\ref{step3} and part (2) of Proposition~\ref{stoldprops}, we must
have $\ord\alpha=r<k$, but (applying Proposition~\ref{order}) we have just shown
$\ord\alpha\ge k$.

Therefore, no $E_\ell$ can have order type $\omega^k$; and therefore neither can
$E$; and therefore neither can $\dft_f(S)$.  Moreover, neither can
\[ \{ \dft(f(3^{n_1},\ldots,3^{n_k})) : (n_1,\ldots,n_k)\in S\}, \]
as it is also covered by finitely many translates of $\dft_f(S)$.
\end{proof}

We now prove a stronger version of Proposition~\ref{step3}, where we bootstrap
its inequalities into equations.  We can now say exactly what the type of limits
we are looking at.

\begin{thm}
\label{precursor}
Let $q$ be a stable number and let $\eta=\delta(q)+k$ for some $k\in \Nn$.  Then
$\eta\in\clDa{\cpx{q}+k}$.  Moreover, if we write $\eta = \clD(\alpha)$, then
$\ord \alpha = k$.  Similarly, if we write $\eta=\clDa{\cpx{q}+k}(\beta)$, then
$\ord\beta = k$.
\end{thm}

\begin{proof}
By Proposition~\ref{exists}, we can choose a substantial polynomial $f$ with
leading coefficient $q$ with $\dft(f)=\eta$ and degree $k$.  Let $S$ be its
exceptional set.  By Proposition~\ref{indivtype}, the order type of
$\dft_f(\Nn^k)$ is $\omega^k$, but by Proposition~\ref{usu}, the order type of
$\dft_f(S)$ is strictly less than $\omega^k$; by Proposition~\ref{cutandpaste},
this implies that $\dft_f(\Nn^k\setminus S)$ has order type $\omega^k$ as well.
Moreover, $\dft_f(\Nn^k\setminus S)$ must obviously be cofinal within
$\dft_f(\Nn^k)$, as otherwise it would have strictly smaller order type;
therefore its supremum is (by Proposition~\ref{dftbd}) equal to $\dft(f)=\eta$.

But for $(n_1,\ldots,n_k)\notin S$, we have, by definition, that
\[
\dft(f(3^{n_1},\ldots,3^{n_k})) = \dft_f(n_1,\ldots,n_k),
\]
and that moreover this defect is a stable one.  Moreover, if
$(n_1,\ldots,n_k)\notin S$, then
\[
\cpx{f(3^{n_1},\ldots,3^{n_k})} = \cpx{f} + 3(n_1+\ldots+n_k) \equiv \cpx{f}
\pmod{3},
\]
and $\cpx{f}=\cpx{q}+k$.

Therefore, the set $\dft_f(\Nn^k\setminus S)$ is a subset of $\Dst^{\cpx{q}+k}$,
has order type $\omega^k$, and has supremum $\eta$.  This shows that
$\eta\in\clDa{\cpx{q}+k}$, and also (by Proposition~\ref{order}) that
$\ord\beta\ge k$ and $\ord\alpha\ge k$.
And we already know from Proposition~\ref{step3}, that $\ord\beta\le k$ and that
$\ord\alpha\le k$, so we conclude that $\ord\alpha = \ord\beta = k$, proving the
theorem.
\end{proof}

We have now essentially done the work of proving the main theorem.  All that
remains is to use order theory and topology to convert Theorem~\ref{precursor}
into more usable forms.

\subsection{From Theorem~\ref{precursor} to Theorem~\ref{cj8}}
\label{corsec}

Having proven Proposition~\ref{usu} and Theorem~\ref{precursor}, we now apply
it to yield a number of corollaries, including Theorem~\ref{cj8} and other
theorems discussed in the introduction.

\begin{prop}
\label{selfsim3}
If $u$ is a congruence class modulo $3$, then
we have \[\clDa{u}'=\clDa{u-1}+1\] and 
$\clDa{u}'''=\clDa{u}+3$.
\end{prop}

\begin{proof}
The set $\clDa{u}$ is closed, so $\clDa{u}'$ is a subset of it.  The question
then is, which points of $\clDa{u}$ are limit points?  By
Proposition~\ref{limitpoints}, they are the
points $\clDa{u}(\alpha)$ with $\ord\alpha\ge 1$.  However, by
Theorem~\ref{precursor} and Proposition~\ref{subclosure}, any such point can be
written as $\delta(q)+k$, where $q$ is stable, $k=\ord \alpha$, and
$u\equiv\cpx{q}+k \pmod{3}$.

Let \[\eta=\clDa{u}(\alpha)-1=\delta(q)+k-1.\]  Since $k\ge 1$, $k-1\ge 0$, and
so by Theorem~\ref{precursor}, \[\eta\in\clDa{\cpx{q}+k-1}=\clDa{u-1}.\]  This
shows that $\clDa{u}'\subseteq \clDa{u-1}+1$.

Conversely, if we start with $\eta\in\clDa{u-1}$, we may similarly write
$\eta=\delta(q)+k$ for some $k\ge 0$ and some stable $q$ with
$u-1\equiv\cpx{q}+k\pmod{3}$.  So $\eta+1=\delta(q)+k+1$.  So by
Theorem~\ref{precursor}, \[\eta+1\in\clDa{\cpx{q}+k+1}=\clDa{u};\] moreover,
since $k+1\ge 1$, we conclude by Theorem~\ref{precursor} and
Proposition~\ref{limitpoints} that it is a limit point of the set.  This proves
that $\clDa{u-1}+1\subseteq\clDa{u}'$, and so $\clDa{u}'=\clDa{u-1}+1$.

The second statement, that $\clDa{u}'''=\clDa{u}+3$, then just follows from
iterating the previous statement three times.
\end{proof}

We can now prove Theorem~\ref{cj8}.

\begin{proof}[Proof of Theorem~\ref{cj8}]
It suffices to prove the case $k=1$, as the more general case follows from
iterating this.

Fix $a$ and consider the sets $\clDa{u}$ and $\clDa{u+1}$.  By
Proposition~\ref{selfsim3}, we have $\clDa{u+1}'=\clDa{u}+1$.  Therefore,
for $1\le \alpha<\omega^\omega$,
\[ \clDa{u+1}'(\alpha) = \clDa{u}(\alpha) + 1.\]
Since the limit points of $\clDa{u+1}$ are by Proposition~\ref{limitpoints}
precisely the points of the form $\clDa{u+1}(\omega\beta)$ for some $1\le
\beta<\omega^\omega$, and since we are $1$-indexing, this means that
\[ \clDa{u+1}(\omega\alpha) = \clDa{u+1}'(\alpha) = \clDa{u}(\alpha) + 1,\]
as desired.
\end{proof}

We can now proceed to prove various corollaries.

We begin with the split-up analogue of Theorem~\ref{closure-intro}.  See
Section~\ref{dftexp} for a discussion of how this result may be interpreted.

\begin{cor}
\label{closure3}
For $u$ a congruence class modulo $3$,
\[\clDa{u}=(\D^u_\st+3\Nn)\cup (\D^{u-1}_\st+3\Nn+1)\cup (\D^{u-2}_\st+3\Nn+2).
\]
Equivalently,
\[\clDa{u} = \{k-3\log_3 n: k\ge \cpx{n},\ k\equiv u\pmod{3}\}.\]
Moreover, each $\D^u_\st$ is a discrete set.
\end{cor}

\begin{proof}
By Proposition~\ref{subclosure}, if $\eta\in\clDa{u}$, then $\eta=\dft(n)+\ell$
for some stable $n$ and some $\ell\ge0$ with $\cpx{n}+\ell\equiv u\pmod{3}$.  So
$\eta=(\cpx{n}+\ell)-3\log_3 n$, and we can take $k=\cpx{n}+\ell$.

Conversely,
if we have $k$ and $n$ with $k\ge\cpx{n}$ and $k\equiv u\pmod{3}$, then in
particular we have $k\ge\cpx{n}_\st$; say $k=\cpx{n}_\st+\ell$.  Now, if we
let $K=K(n)$, and write $n'=3^K n$ and $k'=k+3K$, then we also have
$k'=\cpx{n'}+\ell$, and $k'\equiv k \pmod{3}$.  So by
Proposition~\ref{precursor},
\[k - 3\log_3 n = k'-3\log_3 n' = \dft(n')+\ell \in \clDa{\cpx{n'}+\ell} =
\clDa{k'} = \clDa{k} = \clDa{u}.\]

This proves that
\[\clDa{u} = \{k-3\log_3 n: k\ge \cpx{n},\ k\equiv u\pmod{3}\};\]
in the process, we have also shown that
\[\clDa{u} = \{k-3\log_3 n: k\ge \cpx{n}_\st,\ k\equiv u\pmod{3}\}.\]
The statement
\[\clDa{u}=(\D^u_\st+3\Nn)\cup (\D^{u-1}_\st+3\Nn+1)\cup (\D^{u-2}_\st+3\Nn+2).
\]
then just consists of breaking the latter statement up by congruence class.

Finally, note that if $\eta\in\D^u_\st$, say $\eta=\dft(n)$ with $n$ stable,
then $\eta=\dft(n)+0$, so by Propositions~\ref{step3} and \ref{limitpoints}, it
is not a limit point of $\clDa{u}$, i.e.~not a limit point of $\D^u_\st$.  That
is to say, $\D^u_\st$ contains none of its own limit points; it is a discrete
set.
\end{proof}

\begin{cor}
The different $\clDa{u}$ are disjoint.
\end{cor}

\begin{proof}
This follows immediately from Proposition~\ref{subclosure}.
\end{proof}

We also note the following interpretation of Corollary~\ref{closure3}:

\begin{cor}
\label{dftexprsplit}
For each congruence class $u$ modulo $3$,
\[ \clDa{u} = \{ \dft(E): E\textrm{ a $(1,+,\cdot)$-expression},
\cpx{E}\equiv u\pmod{3}\} \]
\end{cor}

\begin{proof}
This follows immediately from Corollary~\ref{closure3} and
Proposition~\ref{dftexprprop}.
\end{proof}

\begin{cor}
If $u_1$ and $u_2$ are congruence classes modulo $3$, then
\[ \clDa{u_1}+\clDa{u_2} \subseteq \clDa{u_1+u_2}. \]
\end{cor}

\begin{proof}
If $\eta_1\in\clDa{u_1}$ and $\eta_2\in\clDa{u_2}$, then by
Corollary~\ref{closure3}, we may write $\eta_i=k_i-3\log_3 n_i$ where $k_i\ge
\cpx{n_i}$ and $k_i\equiv u_i \pmod{3}$.  Then $\cpx{n_1n_2}\le k_1+k_2$, so
\[ \eta_1+\eta_2 = k_1 + k_2 - 3\log_3(n_1n_2) \in \clDa{u_1+u_2}, \]
where here we have applied Corollary~\ref{closure3} again.
\end{proof}

\begin{cor}
\label{translate3}
If $u$ is a congruence class modulo $3$, then
$\clDa{u}=\overline{\mathscr{D}^u}$.
\end{cor}

\begin{proof}
Since $\D^u_\st\subseteq \D^u$, we immediately have $\clDa{u}\subseteq
\clDax{u}$.  For the reverse, let $\eta=\dft(n)\in\D^u$; then we may write
$\eta=\dft_\st(n)+k$, where $k = \cpx{n}-\cpx{n}_\st \ge 0$.  Since
$\cpx{n}\equiv u\pmod{3}$, this means $\cpx{n}_\st \equiv u-k \pmod{3}$.
Thus $\eta \in \D^{u-k}_\st + k$.  By Corollary~\ref{closure3}, this means
$\eta\in \clDa{u}$.  So $\D^u\subseteq \clDa{u}$ and thus $\clDax{u}\subseteq
\clDa{u}$, i.e., $\clDa{u} = \clDax{u}$.
\end{proof}

\begin{cor}
\label{selfsim}
We have $\clD'=\clD+1$.
\end{cor}

\begin{proof}
By Corollary~\ref{selfsim3},
\[
\clD' = (\clDa{0}\cup\clDa{1}\cup\clDa{2})'
= \clDa{0}'\cup\clDa{1}'\cup\clDa{2}'
= (\clDa{2}+1)\cup(\clDa{0}+1)\cup(\clDa{1}+1)
= \clD+1.
\]
\end{proof}

\begin{rem}
Note that we could have proved Corollary~\ref{selfsim} exactly the same way we
proved Corollary~\ref{selfsim3} (or as part of Corollary~\ref{selfsim3}), rather
than as a separate corollary of it.  But this proof highlights that something
similar will still hold if we put
together the different $\clDa{u}$ in a different way.
For instance, one might use the $R$ function from \cite[Definition 3.9]{intdft},
defined as $R(n)=\frac{n}{E(\cpx{n})}$, where $E(k)$ is the largest number
writable with $k$ ones; this is a transformation of the defect.
Then if one defines $\mathscr{R}=\{ R(n): n\in\N \}$, and define 
$\mathscr{R}^u$ similarly, then these will satisfy
\[
\overline{\mathscr{R}^0}' = \frac{2}{3}\overline{\mathscr{R}^2},\quad
\overline{\mathscr{R}^1}' = \frac{3}{4}\overline{\mathscr{R}^0},\quad
\textrm{and}\ 
\overline{\mathscr{R}^2}' = \frac{2}{3}\overline{\mathscr{R}^1},
\]
and thus $\overline{\mathscr{R}^u}''' = \frac{1}{3}\overline{\mathscr{R}^u}$
for each congruence class $u$, and therefore
$\overline{\mathscr{R}}'''=\frac{1}{3}\overline{\mathscr{R}}$ when put together.
Alternatively, one could normalize things differently and define $\mathscr{A}^i$
to be like the original sets $A_i$ from \cite{Arias}, where instead of
$\frac{n}{E(\cpx{k})}$, we use $\frac{n}{3^{\lfloor{\cpx{n}}/{3}\rfloor}}$.
With this one would similarly obtain
\[
\overline{\mathscr{A}^0}' = \frac{1}{3}\overline{\mathscr{A}^2},\quad
\overline{\mathscr{A}^1}' = \overline{\mathscr{A}^0},\quad
\textrm{and}\ 
\overline{\mathscr{A}^2}' = \overline{\mathscr{A}^1},
\]
and then the same consequences as for $\mathscr{R}$.  However, we will not
include a formal proof of these statements here.
\end{rem}

\begin{cor}
\label{cj8weak}
Given $1\le \alpha<\omega^\omega$ and $k\in \Nn$,
$\clD(\omega^k \alpha)=\clD(\alpha)+k$.
\end{cor}

\begin{proof}
By Proposition~\ref{limitpoints}, we know that the limit points of
$\clD$ are the points of the form $\clD(\omega\beta)$ for some
$1\le\beta<\omega^\omega$.  In other words, $\clD'(\alpha)=\clD(\omega\alpha)$.
So by Corollary~\ref{selfsim},
\[ \clD(\omega\alpha) = \clD'(\alpha) = \clD(\alpha) + 1.\]
Iterating this, we obtain $\clD(\omega^k \alpha) = \clD(\alpha)+k.$
\end{proof}

\begin{cor}
\label{closure-weak}
We have $\clD=\D_\st+\Nn$; equivalently,
\[\clD = \{k-3\log_3 n: k\ge \cpx{n}\}.\]
Moreover, $\D_\st$ is a discrete set.
\end{cor}

\begin{proof}
The first statement follows immediately from Proposition~\ref{subclosure} and
Theorem~\ref{precursor}, and the equivalence of the two forms is obvious.

To prove the second statement, note that if $\eta\in\Dst$, say $\eta=\dft(n)$
with $n$ stable, then $\eta=\dft(n)+0$, so by Propositions~\ref{step3} and
\ref{limitpoints}, it is not a limit point of $\clD$, i.e., not a limit point
of $\Dst$.  That is to say, $\Dst$ contains none of its own limit
points; it is a discrete set.
\end{proof}

\begin{cor}
\label{dftexprcor}
\[ \clD = \{ \dft(E): E\textrm{ a $(1,+,\cdot)$-expression}\} \]
\end{cor}

\begin{proof}
This follows immediately from Corollary~\ref{closure-weak} and
Proposition~\ref{dftexprprop}.
\end{proof}

\begin{cor}
\label{addclosed}
The set $\clD$ is closed under addition.
\end{cor}

\begin{proof}
If $\eta_1, \eta_2 \in\clD$, then by Corollary~\ref{closure-weak}, we may write
$\eta_i=k_i-3\log_3 n_i$ where $k_i\ge \cpx{n_i}$.  Then $\cpx{n_1n_2}\le
k_1+k_2$, so \[ \eta_1+\eta_2 = k_1 + k_2 - 3\log_3(n_1n_2) \in \clD, \] where
here we have applied Corollary~\ref{closure-weak} again.
\end{proof}

\begin{cor}
\label{translate}
We have $\clD=\overline{\D}$, and, given $u$ a congruence class modulo $3$,
$\clDa{u}=\overline{\D^u}$.
\end{cor}

\begin{proof}
Since $\Dst^u\subseteq \D^u$, we immediately have $\clDa{u}\subseteq \clDax{u}$;
and similarly $\clD\subseteq\clDx$.  For the reverse, if $\eta\in\D^u$, that
means there is some $n$ such that $\eta=\cpx{n}-3\log_3 n$ with $\cpx{n}\equiv u
\pmod{3}$; since certainly $\cpx{n}\ge\cpx{n}$ (or since
$\cpx{n}\ge\cpx{n}_\st$), this means that by Proposition~\ref{closure3},
$\eta\in\clDa{u}$.  So $\D^u\subseteq\clDa{u}$ and thus
$\clDax{u}\subseteq\clDa{u}$, and so $\clDax{u}=\clDa{u}$.  The same reasoning
using Proposition~\ref{closure-weak} yields that $\D\subseteq\clD$ and so
$\overline{\D}=\clD$.
\end{proof}

Having proven all this, let us now prove the claims from the introduction.

\begin{proof}[Proofs of Theorems~\ref{selfsim-intro}, \ref{cj8weak-intro},
\ref{closure-intro}, \ref{dftexprthm-intro}, and \ref{closure}]
Theorem~\ref{cj8weak-intro} simply consists of Corollary~\ref{cj8weak} together
with Corollary~\ref{translate}.  Theorem~\ref{selfsim-intro} is simply
Corollary~\ref{selfsim} together with Corollary~\ref{translate}.
Theorem~\ref{closure} follows from Corollary~\ref{closure-weak} together with
the fact that $\Dst\subseteq\D\subseteq \Dst+\Nn$.  Theorem~\ref{closure-intro}
is then simply a weaker version of this, and Theorem~\ref{dftexprthm-intro}
simply combines this with Corollary~\ref{dftexprcor}.
\end{proof}

\section{Applications to self-similarity conjectures and the degree-$1$ case}
\label{cj9sec}

Having now proven all these forms of \cite[Conjecture~8]{Arias},
let us prove something closer to the original form.

\begin{thm}
\label{cj8orig}
For $u$ a congruence class modulo $3$ and $0\le \alpha<\omega^\omega$,
\[\lim_{k\to\infty} \D^{u+1}_\st [\omega\alpha+k]=\D^u_\st [\alpha]+1.\]

Similarly,
\[\lim_{k\to\infty} \D_\st [\omega\alpha+k]=\D_\st [\alpha]+1.\]
\end{thm}

\begin{proof}
From Theorem~\ref{cj8}, we know that
\[\clDa{u+1}(\omega(\alpha+1))=\clDa{u}(\alpha+1)+1.\]
Using Proposition~\ref{closure3} and Proposition~\ref{barshift}, we know that
$\clDa{u}(\alpha+1)=\Da{u}[\alpha]$.  Moreover, since $\clDa{u+1}$ is a closed
set, and we know by Proposition~\ref{topologies} that the subspace topology on
it coincides with the order topology, we have that
\[\clDa{u+1}(\omega(\alpha+1)) = \lim_{k\to\infty} \clDa{u}(\omega\alpha+k).\]
Applying Proposition~\ref{barshift} again, we conclude
\[\clDa{u+1}(\omega(\alpha+1)) = \lim_{k\to\infty} \D^u_\st[\omega\alpha+k],\]
and therefore
\[\lim_{k\to\infty} \D^{u+1}_\st [\omega\alpha+k]=\D^u_\st [\alpha]+1,\]
proving the first claim.  The proof of the second claim is similar.
\end{proof}

We can also immediately prove a weak version of Theorem~\ref{cj2thm}:

\begin{thm}[Weak version of Theorem~\ref{cj2thm}]
\label{cj2weak}
We have:
\begin{enumerate}
\item Suppose $a$ is stable.  Then there exists $K$ such that, for all $k\ge K$
and all $\ell\ge 0$,
\[\cpx{(a3^k+1)3^\ell}=\cpx{a}+3k+3\ell+1 .\]
\item Suppose $ab$ is stable and $\cpx{ab}=\cpx{a}+\cpx{b}$.  Then there exists
$K$ such that, for all $k\ge K$ and all $\ell\ge 0$,
\[\cpx{b(a3^k+1)3^\ell}=\cpx{a}+\cpx{b}+3k+3\ell+1 .\]
\end{enumerate}
\end{thm}

This theorem is the same as Theorem~\ref{cj2thm}, just without the computability
requirement.

\begin{proof}
In either case, we are considering a substantial polynomial $f$ of degree $1$;
in case (1), $f(x) = ax+1$, with $\cpx{f}=\cpx{a}+1$, and in case (2),
$f(x)=b(ax+1)$, with $\cpx{f}=\cpx{a}+\cpx{b}+1$ (by Propositions~\ref{1subst}
and \ref{ineq}).

Let $S$ be the exceptional set of $f$.  Then by Proposition~\ref{usu},
$\dft_f(S)$ has order type less than $\omega$, i.e., is finite; since $\dft_f$
is strictly increasing, this implies $S$ is finite as well.  So, choose $K$ to
be larger than any element of $S$.  Then, for $k\ge K$,
\[ \cpx{f(3^k)3^\ell} = \cpx{f} + 3k + 3\ell \]
by definition of $S$.  Substituting in the particular values of $f$ and
$\cpx{f}$ yields the theorem.
\end{proof}

We will prove the full Theorem~\ref{cj2thm}, with the
computability requirement, shortly in Section~\ref{deg1}.

Now we address Conjectures~9, 10, and 11 from \cite{Arias},
which also deal with the degree $1$ case.  (It is possible to write down
generalizations beyond this case, but these generalizations are not
interesting, and would essentially just recapitulate the contents of
Section~\ref{secproof}.)  Specifically, we prove:

\begin{thm}
\label{cj9corsep}
Let $n$ be a number not divisible by $3$, with $\cpx{n}_\st\equiv u\pmod{3}$,
and write $\dft_\st(n)=\Da{u}[\alpha]$.
Then the set
\[\{\Da{u+1}[\omega\alpha+k]: k\in\Nn\},\]
which may be equivalently written as
\[\Da{u+1} \cap [\Da{u+1}[\omega\alpha], \dft_\st(n)+1),\]
has finite symmetric
difference with the set
\[
\{ \dft_\st(b(a3^k + 1)):
\ k\ge 0,\ ab=n,\ \cpx{n}_\st =\cpx{a}_\st+\cpx{b}_\st\}.
\]

Similarly, if instead $\dft_\st(n)=\Dst[\alpha]$, then the set $\{\Dst[\omega\alpha+k]: k\in\Nn\}$, which may equivalently
be written as $\Dst\cap[\Dst[\omega\alpha],\dft_\st(n)+1)$, has finite symmetric
difference with the set
\[
\{ \dft_\st(b(a3^k + 1)):
\ k\ge 0,\ ab=n,\ \cpx{n}_\st =\cpx{a}_\st+\cpx{b}_\st\}.
\]
\end{thm}

This is intended as a repair of \cite[Conjectures~9--11]{Arias}, whose original
statements are a bit too strong.  These conjectures state that
if $q$ is a stable number with $\dft(q)=\Da{u}[\alpha]$, then the set
$A := \{\Da{u+1}[\omega\alpha+k]: k\in\Nn\}$ has finite symmetric difference
with the set
\[
E := \{ \dft_\st(b(a3^k + 1)):
\ k\ge 0,\ ab=q \},
\]
with the one-sided difference $A\setminus E$ being a finite subset of
$\{\dft_\st(2^k): k\in\N\}$.

(Here we have rephrased these conjectures somewhat from their original language;
see the Appendix of \cite{paperwo} for more on translating between these two
frameworks.)

As you can see from comparison to Theorem~\ref{cj9corsep}, the
removal of the final clause is the only major alteration.

Let us present counterexamples to the final clause for all three congruence
classes.  For the case where $\cpx{q}\equiv0\pmod{3}$, we can consider $q=64 =
\Da{0}[2]$, and observe that $70=\Da{1}[\omega2+1]$, even though $70$ cannot be
written as $b(a3^k+1)3^\ell$ with $ab=64$.  For the case where
$\cpx{q}\equiv1\pmod{3}$, we can consider $q=32=\Da{1}[1]$, and observe that
$35=\Da{2}[\omega+1]$, even though $35$ cannot be written as $b(a3^k+1)3^\ell$
with $ab=32$.  And for the case $\cpx{q}\equiv 2\pmod{3}$, we can consider
$q=5=\Da{2}[2]$, and observe that $1280=\Da{0}[\omega2]$, even though $1280$
cannot be written as $5(3^k+1)3^\ell$ or as $(5\cdot3^k+1)3^\ell$.  These
counterexamples have been chosen to have minimal defect.  Note that we have
omitted the computations to verify these, but all of this may be easily verified
from a good cover of $\overline{B}_{14\dft(2)}$, which can be computed using the
algorithms from \cite{paperalg}.

Now, it is worth noting that here we look not at factorizations satisfying
$\cpx{ab}=\cpx{a}+\cpx{b}$, but rather $\cpx{ab}_\st=\cpx{a}_\st+\cpx{b}_\st$.
Of course, if all divisors of $n$ (other than $1$) are stable, then these two
conditions are the same (aside from the cases where $a=1$ or $b=1$).  But in
general they may not be.  For instance, the stable number $856$ can be factored
as $8\cdot 107$.  We do not have $\cpx{8}+\cpx{107}=\cpx{856}$, but we do have
$\cpx{8}_\st+\cpx{107}_\st = \cpx{856}_\st$, and it is important that we do not
exclude this case.  (The stability of $856$, and the various values of
$\cpx{n}_\st$, were verified using the algorithms from \cite{paperalg}).  See
Section~7 of \cite{paperalg} for more information on the relation between these
two conditions on factorizations.

Note also that the condition that $n$ is not divisible by $3$ is not essential;
we include it because the theorem does not distinguish between a number $n$ and
$n3^k$ for $k\in \mathbb{Z}$, so we have chosen to require that $n$ is not
divisible by $3$ in order to keep things concrete and canonical.  No generality
is lost in this way.

In addition, while we have not stated this theorem constructively, it is
possible to prove it in a constructive manner.  We will skip doing this here
because we do not expect getting effective numbers for this to be of much
relevance, compared to the theorems in the next section where effectivity may be
of more use.

We now prove the theorem.

\begin{proof}[Proof of Theorem~\ref{cj9corsep}]
We include only a proof of the first part, as the proof of the second part is
exactly analogous.

Let $q=n3^{K(n)}$, so that $q$ is stable and $\dft(q)=\dft_\st(n)$.
So we know from Theorem~\ref{cj8orig} that
\[\lim_{k\to\infty} \D^{u+1}_\st [\omega\alpha+k]=\dft(q)+1=\dft_\st(n)+1.\]

Choose $a$ and $b$ with $ab=n$ and $\cpx{n}_\st=\cpx{a}_\st+\cpx{b}_\st$.
Let $A=a3^{K(a)}$, $B=b3^{K(b)}$, and $N=AB$.  Then we have 
\[ \cpx{N}_\st=\cpx{n}_\st+3K(a)+3K(b)
= \cpx{a}_\st+3K(a)+\cpx{b}_\st+3K(b)
= \cpx{A}_\st + \cpx{B}_\st,\]
so by Proposition~\ref{goodfac}, $N$ is also stable, and so
\[ \cpx{N}=\cpx{A}+\cpx{B}. \]

By Theorem~\ref{cj2weak}, we know that all but finitely many $k$ satisfy
\[ \cpx{B(A3^k+1)}_\st = \cpx{N} + 3k + 1 \]
and therefore
\[ \dft_\st(B(A3^k+1)) =
\cpx{N}+1-3\log_3 (N+B3^{-k}),
\]
meaning
\[
\lim_{k\to\infty} \dft_\st(B(A3^k+1)) = \cpx{N}+1-3\log_3 N = \dft(N)
+1=\dft_\st(n)+1.
\]
Moreover,
\[ \dft_\st(B(A3^k+1)) = \dft_\st(b(A3^k+1)), \]
and for all $k\ge K(a)$,
\[ \dft_\st(b(a3^k+1)) = \dft_\st(b(A3^{k-K(a)}+1)). \]
So in fact,
\[
\lim_{k\to\infty} \dft_\st(b(a3^k+1)) = \dft_\st(n)+1.
\]

Since for $k\ge K(a)$ we have $\dft_\st(b(a3^k+1))\in \Da{\cpx{q}+1}=\Da{u+1}$,
this means that all but finitely many of the $\dft_\st(b(a3^k+1))$ must be
numbers of the form $\Da{u+1}[\omega\alpha+k]$.

For the converse, take a good covering $\sS$ of $\overline{B}_{\dft(q)+1}$.
Let
\[\zeta = \max (\{ \dft(f,C): (f,C)\in \sS,\ \dft(f,C)<\dft(q)+1 \} \cup
\{0\}).\]
Now, since
\[\lim_{k\to\infty} \D^{u+1}_\st [\omega\alpha+k]=\dft(q)+1,\]
all but finitely many $\Da{u+1}[\omega\alpha+k]$ must lie in
$(\zeta,\dft(q)+1)$.  This means
they are of the form $\dft(m)$ for some leader $m$ that is efficiently
$3$-represented by some $(f,C)\in \sS$ (note we must actually have $C=\cpx{f}$).
But also,
\[\zeta<\dft(m)\le \dft(f,C)\le \dft(q)+1,\]
so we must have $\dft(f,C)=\dft(q)+1$.  By Proposition~\ref{cj9prop}, this
implies $\deg f\le 1$.  But we cannot have $\deg f=0$ as then we would have
$\dft(m)=\dft(f,C)$, in contradiction to the assumption that
$\dft(m)\in(\zeta,\dft(q)+1)$.  So $\deg f=1$, which again by
Proposition~\ref{cj9prop}, means $f$ is substantial.

So, by Proposition~\ref{1subst}, this means $f$ either has the form
$f(x)=B(Ax+1)$, with $AB$ stable and $\cpx{AB}=\cpx{A}+\cpx{B}$, and
$\cpx{f}=\cpx{AB}+1$; or $f(x)=Ax+1$, with $A$ stable, and $\cpx{f}=\cpx{A}+1$.
In this latter case, let $B=1$.  Since $\dft(f)=\dft(q)+1$, and differing stable
defects cannot be congruent modulo $1$, we must have $\dft(AB)=\dft(q)$; so
$AB=q3^j$ for some $j\in\mathbb{Z}$, and $\cpx{AB}=\cpx{q}+3j$.

Now, if $j<0$, then we may define $g(x)=B(A3^{-j}x+1)$, and $D=C-3j$, with the
result that all but finitely many $m$ that are efficiently $3$-represented by
$(f,C)$ will also be efficiently $3$-represented by $(g,D)$; and if we consider
$g$ instead of $f$, then we will have $B(A3^{-j})=q$ exactly.  So it suffices to
consider the case where $j\ge 0$, as otherwise we may replace $f$ by $g$ to
obtain $j=0$, at the loss of only finitely many defects.

So now we have that all but finitely many of our defects are of the form
$\dft(m)$, where $m$ is of the form $B(A3^k+1)$, $AB=q3^j$ is stable, and either
$\cpx{A}+\cpx{B}=\cpx{q3^j}$ or $B=1$.  By Theorem~\ref{cj2weak} again, all but
finitely many $B(A3^k+1)$ are stable, and we have only finitely many ordered
pairs $(A,B)$, meaning all but finitely many of our defects are of the form
$\dft_\st(B(A3^k+1))$.

Now let $a$ be the part of $A$ that is not divisible by $3$, and $b$ be the
part of $B$ that is not divisible by $3$, so that $ab=n$.  Moreover, since
either $B=1$ or $\cpx{A}+\cpx{B}=\cpx{q}$, we have
\[\cpx{a}_\st+\cpx{b}_\st=\cpx{n}_\st.\]  And by above all
but finitely many of our defects are of the form 
$\dft_\st(B(A3^k+1))$, but this is the same as
$\dft_\st(b(a3^{k+k_0}+1))$, where $A=a3^{k_0}$.

So, all but finitely many of our defects have the required form.
\end{proof}

\subsection{Stability and computability of the degree-$1$ case}
\label{deg1}

We now investigate the implications of Proposition~\ref{usu} and
Theorem~\ref{cj9corsep} for low-defect polynomials of degree $1$, this time
taking computability into account.  The simplest
case is Theorem~\ref{cj2thm}.

Of course, if we ignore the part about computability, then we could prove
Theorem~\ref{cj2thm} by direct application of Proposition~\ref{usu}; indeed, we
already did this as Theorem~\ref{cj2weak}.  Since we want a slightly stronger
statement, however, we will have to do slightly more work.

\begin{proof}[Proof of Theorem~\ref{cj2thm}]
In either case, we are considering a substantial polynomial $f$ of degree $1$;
in case (1), $f(x) = ax+1$, with $\cpx{f}=\cpx{a}+1$, and in case (2),
$f(x)=b(ax+1)$, with $\cpx{f}=\cpx{a}+\cpx{b}+1$ (by Propositions~\ref{1subst}
and \ref{ineq}).

So, let $q$ be the leading coefficient of $f$ (which is $a$ in case (1) and $ab$
in case (2)), and let $\eta = \dft(q)$, so $\dft(f) = \eta + 1$.  Now, given any
$r\in\Nn$, we can, by Theorem~\ref{goodcomput},
compute a good covering $\sS_r$ of $\overline{B}_{\eta-r}$.

So given
$0\le r\le\lfloor \eta\rfloor$, let
\[\zeta_r = \max (\{ \dft(g,C): (g,C)\in \sS_r,\ \dft(g,C)<\eta-r \} \cup
\{0\}).\]
Now if we had a number $n$ with $\dft(n)\in (\zeta_r, \eta-r)$, then $n$ would
be efficiently $3$-represented by $(\xpdd{g},C)$ for some $(g,C)\in \sS_r$; this
would imply $\dft(n)\le\dft(g,C)\le \eta-r$ and therefore $\dft(g,C)=\eta-r$.
However, by Proposition~\ref{polydft}, $\dft(g,C)$ is equal to a stable defect
plus a nonnegative integer, and as such not equal to any $\eta-r$ except
possibly when $r=0$;
therefore no such $n$ can exist unless $r=0$.  But if $r=0$, then
$\dft(g,C)=\eta$.  Since $\eta$ is a stable defect, this forces $\deg g = 0$ by
Proposition~\ref{cj9prop} again.  But this means that $\dft(n)=\eta-r$, in
contradiction to the assumption that $\dft(n)<\eta-r$.  Therefore, whether or
not $r>0$, we obtain $\D \cap (\zeta_r, \eta-r)=\emptyset$.

So, for each $0\le r\le \lfloor\eta\rfloor$, compute $K_r$ such that
$\dft_f(K_r)-r-1 > \zeta_r$; this is
straightforward as (by Definition~\ref{fdftdef})
\[\dft_f(k)=\cpx{a}+1-3\log_3(a+3^{-k})\] in case (1) and
\[\dft_f(k)=\cpx{ab}+1-3\log_3(b(a+3^{-k}))\] in case (2).
(Note that either way, $\lim_{k} \dft_f(k)=\dft(f)=\eta-r>\zeta_r$.)
Now let \[K = \max_{0\le r\le\lfloor \eta\rfloor} K_r.\]  Then for $k\ge K$, we
know $\dft_\st(f(3^k))\equiv\dft_f(k)\pmod {1}$, and also
\[\dft_\st(f(3^k))\le\dft_f(k)<\dft(f)=\eta+1.\]

This means we must have $\dft_\st(f(3^k))=\dft_f(k)$, because otherwise we would
have $\dft_\st(f(3^k)) = \dft_f(k)-1-r$ for some $r\ge 0$; but we know
\[ \zeta_r < \dft_f(K_r)-r-1 \le \dft_f(k)-r-1
< \dft(f) - 1 -r = \eta-r, \]
i.e., $\dft_f(k)-1-r\in(\zeta_r,\eta-r)$, while $\dft_\st(f(3^k))\in\Dst$, so
these quantities cannot be equal as these sets are disjoint.

Therefore, we conclude that for such $k$, we have $\cpx{f(3^k)}_\st =
\cpx{f}+3k$.  Or, in other words, for $k\ge K$ and $\ell\ge 0$, we have
$\cpx{\xpdd{f}(3^k, 3^\ell)}=\cpx{f}+3k+3\ell$.  Recalling once again that in
case (1) we have $\cpx{f}=\cpx{a}+1$ and in case (2) we have
$\cpx{f}=\cpx{a}+\cpx{b}+1$, and noting that all the steps in determining $K$
were computable, this proves the theorem.
\end{proof}

As was mentioned in Section~\ref{secproof}, this proof raises the question of
whether one can show for substantial polynomials of higher degree whether the
exceptional set can be shown to be ``small'' in a stronger sense than that
implied by Proposition~\ref{usu}, and in a future paper \cite{hyperplanes} we
shall show that it can be.

Returning to the degree $1$ case, however, we can immediately write down the
following corollary of Theorem~\ref{cj2thm}:

\begin{cor}
\label{cj2cor}
We have:
\begin{enumerate}
\item Let $a$ be a natural number.  Then there exists $K$ such that, for all
$k\ge K$ and all $\ell\ge 0$,
\[\cpx{(a3^k+1)3^\ell}=\cpx{a}_\st+3k+3\ell+1 .\]
\item Suppose $a$ and $b$ are natural numbers and
$\cpx{ab}_\st=\cpx{a}_\st+\cpx{b}$.  Then there exists $K$ such that, for all
$k\ge K$ and all $\ell\ge 0$,
\[\cpx{b(a3^k+1)3^\ell}=\cpx{a}_\st+\cpx{b}+3k+3\ell+1 .\]
\item Suppose $a$ and $b$ are natural numbers and
$\cpx{ab}_\st=\cpx{a}_\st+\cpx{b}_\st$.  Then there exists $K$ such that,
for all $k\ge K$ and all $\ell\ge K(b)$,
\[\cpx{b(a3^k+1)3^\ell}=\cpx{a}_\st+\cpx{b}_\st+3k+3\ell+1 .\]
\end{enumerate}
Moreover, in all these cases, it is possible to algorithmically compute how
large $K$ needs to be.
\end{cor}

\begin{proof}
For part (1), pick $k_0$ such that $a3^{k_0}$ is stable; let $A=a3^{k_0}$.
(Note that by Theorem~\ref{computk}, $k_0$ can be computed from $a$.)  Then
by part (1) of Theorem~\ref{cj2thm}, for all sufficiently large $k$ (and we can
compute how large from $k_0$ and $A$), we have
\begin{multline*}
\cpx{(a3^k + 1)3^\ell} =
\cpx{(A 3^{k-k_0} +1)3^\ell} = \\
\cpx{A} + 3(k-k_0) + 3\ell + 1 =
\cpx{a}_\st + 3k + 3\ell + 1.
\end{multline*}

For part (2), pick $k_0$ large enough such that $a3^{k_0}$ and $ab3^{k_0}$ are
both stable; let $A=a3^{k_0}$.  Again, note that $k_0$ may be computed from $a$
and $b$.
Then $Ab$ is stable, and
\[\cpx{Ab} = \cpx{ab}_\st + 3k_0
= \cpx{a}_\st + \cpx{b} + 3k_0
= \cpx{A} + \cpx{b},\]
so we may apply part (2) of Theorem~\ref{cj2thm}.  So for all sufficiently large
$k$ (and we can compute how large from $A$, $b$, and $k_0$), we have
\begin{multline*}
\cpx{b(a3^k + 1)3^\ell} =
\cpx{b(A 3^{k-k_0} +1)3^\ell} = \\
\cpx{A} + \cpx{b} + 3(k-k_0) + 3\ell + 1 =
\cpx{a}_\st + \cpx{b} + 3k + 3\ell + 1.
\end{multline*}

Finally, for part (3), let $\ell_0=K(b)$, and pick $k_0$ large enough such that
both $a3^{k_0}$ and $ab3^{k_0+\ell_0}$ are both stable; again, these quantities
can be computed from $a$ and $b$.  Let $A=a3^{k_0}$ and
let $B=b3^{\ell_0}$.
Then $AB$ is stable, and
\[\cpx{AB} = \cpx{ab}_\st + 3k_0 + 3\ell_0
= \cpx{a}_\st + \cpx{b}_\st + 3k_0 + 3\ell_0
= \cpx{A} + \cpx{B},\]
so we may again apply part (2) of Theorem~\ref{cj2thm}.  So for all sufficiently
large $k$ (and how large can be computed from $A$, $B$, and $k_0$), and any
$\ell\ge \ell_0$, we have
\begin{multline*}
\cpx{b(a3^k + 1)3^\ell} =
\cpx{B(A 3^{k-k_0} +1)3^{\ell-\ell_0}} = \\
\cpx{A} + \cpx{B} + 3(k-k_0) + 3(\ell-\ell_0) + 1 =
\cpx{a}_\st + \cpx{b}_\st + 3k + 3\ell + 1.
\end{multline*}
\end{proof}

However, this is just a corollary.  In fact, it is possible to go further than
this, and prove Theorem~\ref{dragons}, which requires venturing out of the realm
of substantial polynomials and into the realm of polynomials that are, one might
say, just barely insubstantial.  To do this, we use the idea of
Theorem~\ref{cj9corsep}, even though that exact statement will not appear
in the proof.

Now one might say that the polynomials considered in Theorem~\ref{dragons} have
``insubstantiality $1$'', but we will not actually attempt to define a general
notion of ``insubstantiality'', as it is not entirely clear how to do that.  As
noted in Section~\ref{secdragon}, however, one cannot extend
Theorem~\ref{dragons} to cases of ``insubstantiality $2$'', as demonstrated by
the case of $2(1094x+1)$.  (Note that the numbers $2$, $1094$, and $2188$ are
all stable, as can be computed with the algorithms from \cite{paperalg}.)

Let us make another note about this theorem before we prove it.  The division of
Theorem~\ref{dragons} into two parts, both of which require $a$ to be stable and
the second of which requires $b$ to be stable, may make it seem that something
has been lost in moving to barely-insubstantial polynomials; after all,
Theorem~\ref{cj2thm} has no stability condition on $a$ or $b$, only on $ab$.
However, that is because in Theorem~\ref{cj2thm}, the stability conditions on
$a$ and $b$ are implicit.  By Proposition~\ref{goodfac}, if $ab$ is stable and
$\cpx{ab}=\cpx{a}+\cpx{b}$, then $a$ and $b$ are themselves stable.  But in
Theorem~\ref{dragons}, we do not have $\cpx{ab}=\cpx{a}+\cpx{b}$, but rather
$\cpx{ab}=\cpx{a}+\cpx{b}-1$.  This requires adding explicit stability
conditions on $a$ and $b$, where in Theorem~\ref{cj2thm} they were implicit.

We now prove Theorem~\ref{dragons}.

\begin{proof}[Proof of Theorem~\ref{dragons}]
Suppose $ab$ is stable, $\cpx{a}+\cpx{b}=\cpx{ab}+1$, $b>1$, and $a$ is stable.
Let $q=ab$, and let $f(x)=b(ax+1)$ and
\[C=\cpx{a}+\cpx{b}+1=\cpx{q}+2,\] so
$\dft(f,C)=\dft(q)+2$.  Let \[\eta = \dft(q) + 1 = \dft(f,C) - 1.\]

We wish to find a $K$ such that, for all $k\ge K$, we have
$\cpx{f(3^k)}=C+3k$; and so that, if $b$ is stable, we moreover have
$\cpx{f(3^k)}_\st = C+3k$.  Now, if the first of these statements fails, then we
obtain
\[ \cpx{f(3^k)} \le C + 3k - 1, \]
and if the second fails we obtain
\[ \cpx{f(3^k)}_\st \le C + 3k - 1. \]
These in turn imply
\[ \dft(f(3^k)) < \dft(f,C) - 1 = \eta \]
and
\[ \dft_\st(f(3^k)) < \dft(f,C) - 1 = \eta, \]
respectively.

So, we will find a $K$ such that, for $k\ge K$, we can rule out the first of
these possibilities; and such that, under the assumption that $b$ is stable, we
can rule out the second as well.

Now, given any
$r\in \Nn$, we can, by Theorem~\ref{goodcomput},
compute a good covering $\sS_r$ of $\overline{B}_{\eta-r}$.
So, given $0\le r\le\lfloor \eta\rfloor$, let
\[\zeta_r = \max (\{ \dft(g,D): (g,D)\in \sS_r,\ \dft(g,D)<\eta-r \} \cup
\{0\}).\]
Now if we had a number $n$ with $\dft(n)\in (\zeta_r, \eta-r)$, then $n$ would
be efficiently $3$-represented by some $(\xpdd{g},D)$ for $(g,D)\in \sS_r$; this
would imply $\dft(n)\le\dft(g,D)\le \eta-r$ and therefore $\dft(g,D)=\eta-r$.
However, as before (again using Proposition~\ref{polydft}), $\dft(g,D)$ is equal
to a defect plus a nonnegative integer, and as such not equal to any $\eta-r$
except possibly when $r\le 1$, since $\eta$ is equal to $1$ plus a stable
defect.  So for $r>1$, we have $\Dst\cap (\zeta_r, \eta-r)=\emptyset$.

Moreover, when $r=1$, we have $\dft(g,D)=\eta-1=\dft(q)$.  Since this is a
stable defect, this once again forces $\deg g = 0$.  So this means that
$\dft(n)=\eta-r$, in contradiction to the assumption that $\dft(n)\in
(\zeta_r,\eta-r)$, and we again obtain $\D\cap(\zeta_r, \eta-r)=\emptyset$.

So for each $r>0$, we can as before compute $K_r$ such that $\dft_{f,C}(K_r)-r-1
> \zeta_r$ (since $\lim_k \dft_{f,C}(k)=\dft(f,C)=\eta+1$); we can then be
assured that, for $k\ge K_r$, we cannot have
$\dft(f(3^k))\in(\zeta_r,\eta-r)$ nor can we have
$\dft_\st(f(3^k))\in(\zeta_r,\eta-r)$. This leaves the problem of determining a
suitable $K_0$ for the case of $r=0$.

We claim that in this case, we may pick $K_0$ in the same way; that is, it
suffices to choose $K_0$ such that $\dft_{f,C}(K_0)-1 > \zeta_0$.  In other
words, we wish to show that for $k\ge K_0$ it is not possible to have
$\dft(f(3^k))\in (\zeta_0,\eta)$; and that if $b$ is stable, it is not possible
to have $\dft_\st(f(3^k))\in (\zeta_0,\eta)$.

Now in this $r=0$ case, if $\dft(n)\in (\zeta_0, \eta)$, then we can as before
take $(g,D)\in\sS_0$ that efficiently $3$-represents $n$.  By
Proposition~\ref{cj9prop}, as $\eta=\dft(q)+1$, this implies $\deg g\le 1$, with
$\deg g=1$ if and only if $g$ is substantial.  However, by the same reasoning as
in the $r=1$ case, we cannot have $\deg g=0$, as this would force
$\dft(n)=\eta$.  So $g$ must be a substantial polynomial of degree $1$.
Such a $g$ takes the form
$g(x)=d(cx+1)$, where $cd=q$ and either $\cpx{c}+\cpx{d}=\cpx{q}$ or $d=1$.

So what we wish to show is that we cannot have $\dft(f(3^k))=\dft(g(3^\ell))$
for any $\ell\in\Nn$; and that, if $b$ is stable, we moreover cannot have
$\dft_\st(f(3^k))=\dft(g(3^\ell))$ for any $\ell\in\Nn$.  So, again, assume that
such an equality does hold.

In the former case, $f(3^k)$ is efficiently $3$-represented by our
$(\xpdd{g},D)$; in other words, $f(3^k)=g(3^\ell)3^j$ for some substantial $g\in
\sS_0$ and $j\in\Nn$, with $g(3^\ell)$ a leader.  In the latter case, where we
assume $b$ stable, we merely obtain that $f(3^k)=g(3^\ell)3^j$ for some $j\in
\Z$.  So let us combine these assumptions and say that either $j\ge 0$ or $b$ is
stable, and from this derive a contradiction.

So, in either case, we may write
\[ b(a3^k + 1) = d(c3^\ell+1)3^j, \]
which we may rewrite as
\[ q3^k + b = q3^{\ell+j} + d3^j.\]
Now, we know that $b,d\le q$, and so in particular $b\le q3^k$ and $d3^j \le
q3^{\ell+j}$.  Therefore, we know that
\[ q3^k \le q3^k + b \le 2q3^k \]
and
\[ q3^{\ell+j} \le q3^{\ell+j} + d3^j \le 2q3^{\ell+j}. \]
Since the quantities being bounded are equal, both sets of bounds must apply to
this one quantity, which is only possible if $k=\ell+j$, as otherwise the
described intervals are disjoint.

So $q3^k + b = q3^k + d3^j$, or in other words, $b=d3^j$, which implies
$a=c3^{-j}$, since  $ab=cd$.  Now, if $d=1$, then $c=q$, so $b=3^j$ (and
therefore $j>0$, as $b>1$) and $a=q3^{-j}$.  But we assumed $a$ is stable, so
\[\cpx{q}=\cpx{a3^j}=\cpx{a}+3j=\cpx{a}+\cpx{b},\] contrary to the assumption
that $\cpx{q}=\cpx{a}+\cpx{b}-1$.

So we instead must have $d>1$, $\cpx{c}+\cpx{d}=\cpx{q}$, and therefore $c$ and
$d$ both stable by Proposition~\ref{goodfac}.  Then in this case note that since
$a$ and $c$ are both stable, we have $\cpx{a}=\cpx{c}-3j$ regardless of the sign
of $j$.

Now, here is where we make use of the alternative we set up above, that
either $j\ge 0$ or $b$ is stable.
If $j\ge 0$, then we may conclude that $\cpx{b}=\cpx{d}+3j$, because $d$ is
stable.  While if $b$ is stable, then that means $b$ and $d$ are both stable; so
we may again conclude that $\cpx{b}=\cpx{d}+3j$, using the stability of $d$ if
$j\ge 0$ and using the stability of $b$ if $j\le 0$.  But this means that
\[\cpx{a}+\cpx{b}=\cpx{c}+\cpx{d}=\cpx{q},\] again contrary to the assumption
that $\cpx{a}+\cpx{b}=\cpx{q}+1$, and we have reached a contradiction.

This shows that our choice of $K_0$ satisfies the required conditions; it is not
possible to have $\dft(f(3^k))\in (\zeta_0,\eta)$, and if $b$ is stable, it is
not possible to have $\dft_\st(f(3^k))\in (\zeta_0,\eta)$.

So we may once again let \[K = \max_{0\le r\le\lfloor \eta\rfloor} K_r.\]  Then
for $k\ge K$, we know $\dft(f(3^k))\equiv\dft_{f,C}(k)\pmod {1}$, and also
\[\dft(f(3^k))\le\dft_{f,C}(k)<\dft(f)=\eta+1;\]
and the same is true of $\dft_\st(f(3^k))$.

So once again we conclude that $\dft(f(3^k))=\dft_{f,C}(k)$, and that if $b$ is
stable then $\dft_\st(f(3^k))=\dft_{f,C}(k)$, because otherwise
the defect under consideration would equal $\dft_{f,C}(k)-1-r$ for some $r\ge
0$, and therefore lie in some interval $(\zeta_r,\eta-r)$, a possibility we have
just ruled out (unconditionally for $\dft(f(3^k))$ and under the condition that
$b$ is stable for $\dft_\st(f(3^k))$).

Therefore, we conclude that for $k\ge K$, we have $\cpx{f(3^k)} = C+3k$, and if
$b$ is stable, $\cpx{f(3^k)}_\st=C+3k$.  Applying the definitions of $f$, $C$,
and stable complexity yields the desired equations.  Finally, we note that all
the steps in determining $K$ were computable, so this proves the theorem.
\end{proof}

Finally, just as we generalized Theorem~\ref{cj2thm} to Corollary~\ref{cj2cor},
let us generalize Theorem~\ref{dragons} in the same way.

\begin{cor}
\label{cordragons}
We have:
\begin{enumerate}
\item Suppose $\cpx{a}_\st+\cpx{b}=\cpx{ab}_\st+1$, and $b\ne 1$.  Then there
exists $K$ such that for all $k\ge K$,
\[\cpx{b(a3^k+1)}=\cpx{a}_\st+\cpx{b}+3k+1.\]
\item Suppose $\cpx{a}_\st+\cpx{b}_\st=\cpx{ab}_\st+1$.  Then there exists $K$
such that for all $k\ge K$ and $\ell\ge K(b)$,
\[\cpx{b(a3^k+1)3^\ell}=\cpx{a}_\st+\cpx{b}_\st+3k+3\ell+1.\]
\end{enumerate}
Moreover, in both these cases, it is possible to algorithmically compute how
large $K$ needs to be.
\end{cor}

\begin{proof}
For part (1), pick $k_0$ large enough such that $a3^{k_0}$ and
$ab3^{k_0}$ are both stable; let $A=a3^{k_0}$.  (Note that by
Theorem~\ref{computk}, $k_0$ can be computed from $a$ and $b$.)
Then
\[\cpx{Ab} = \cpx{ab}_\st + 3k_0
= \cpx{a}_\st + \cpx{b} + 3k_0 - 1
= \cpx{A} + \cpx{b} - 1,\]
so we may apply part (1) of Theorem~\ref{dragons}.  So for all sufficiently
large $k$ (and how large can be computed from $A$, $b$, and $k_0$), we have
\begin{multline*}
\cpx{b(a3^k + 1)} =
\cpx{b(A 3^{k-k_0} +1)} = \\
\cpx{A} + \cpx{b} + 3(k-k_0) + 1 =
\cpx{a}_\st + \cpx{b} + 3k + 1.
\end{multline*}

For part (2), let $\ell_0=K(b)$, and pick $k_0$ large enough so that $a3^{k_0}$
and $ab3^{k_0+\ell_0}$ are both stable; let $A=a3^{k_0}$ and $B=b3^{\ell_0}$.
Again, all these quantities may be computed from $a$ and $b$.
Then
\[\cpx{AB} = \cpx{ab}_\st + 3k_0 + 3\ell_0
= \cpx{a}_\st + \cpx{b}_\st + 3k_0 + 3\ell_0 - 1
= \cpx{A} + \cpx{B} - 1,\]
so we may apply part (2) of Theorem~\ref{dragons}.  So for all sufficiently
large $k$ (and how large can be computed from $A$, $B$, and $k_0$), and all
$\ell\ge\ell_0$, we have
\begin{multline*}
\cpx{b(a3^k + 1)3^\ell} =
\cpx{B(A 3^{k-k_0} +1)3^{\ell-\ell_0}} = \\
\cpx{A} + \cpx{B} + 3(k-k_0) + 3(\ell-\ell_0) + 1 =
\cpx{a}_\st + \cpx{b}_\st + 3k + 3\ell + 1.
\end{multline*}
\end{proof}

\subsection*{Acknowledgements}
Thanks to Jeffrey Lagarias for extensive help with editing.  Work of the authors
was supported by NSF grants DMS-0943832 and DMS-1101373.

\appendix

\section{Comparison to addition chains}
\label{addchains}

It is worth discussing here the possibility of theorems analogous to
Theorem~\ref{selfsim-intro} and Theorem~\ref{cj8weak-intro} for addition chains.
An \emph{addition chain} for $n$ is defined to be a sequence
$(a_0,a_1,\ldots,a_r)$ such that $a_0=1$, $a_r=n$, and, for any $1\le k\le r$,
there exist $0\le i, j<k$ such that $a_k = a_i + a_j$; the number $r$ is called
the length of the addition chain.  The shortest length among addition chains for
$n$, called the \emph{addition chain length} of $n$, is denoted $\acl(n)$.
Addition chains were introduced in 1894 by H.~Dellac \cite{Dellac} and
reintroduced in 1937 by A.~Scholz \cite{aufgaben}; extensive surveys on the
topic can be found in Knuth \cite[Section 4.6.3]{TAOCP2} and Subbarao
\cite{subreview}.

The notion of addition chain length has obvious similarities to that of integer
complexity; each is a measure of the resources required to build up the number
$n$ starting from $1$.  Both allow the use of addition, but integer complexity
supplements this by allowing the use of multiplication, while addition chain
length supplements this by allowing the reuse of any number at no additional
cost once it has been constructed.  Furthermore, both measures are approximately
logarithmic; the function $\acl(n)$ satisfies
\[ \log_2 n \le \acl(n) \le 2\log_2 n. \]
 
A difference worth noting is that unlike integer complexity, there is no known
way to compute addition chain length via dynamic programming.  Specifically, to
compute integer complexity this way, one may use the fact that for any $n>1$,

\begin{displaymath}
\cpx{n}=\min_{\substack{a,b<n\in \mathbb{N} \\ a+b=n\ \mathrm{or}\ ab=n}}
	\cpx{a}+\cpx{b}.
\end{displaymath}

By contrast, addition chain length seems to be harder to compute.  Suppose we
have a shortest addition chain $(a_0,\ldots,a_{r-1},a_r)$ for $n$; one might
hope that $(a_0,\ldots,a_{r-1})$ is a shortest addition chain for $a_{r-1}$, but
this need not be the case.  An example is provided by the addition chain
$(1,2,3,4,7)$; this is a shortest addition chain for $7$, but $(1,2,3,4)$ is not
a shortest addition chain for $4$, as $(1,2,4)$ is shorter.   Moreover, there is
no way to assign to each natural number $n$ a shortest addition chain
$(a_0,\ldots,a_r)$ for $n$ such that $(a_0,\ldots,a_{r-1})$ is the addition
chain assigned to $a_{r-1}$ \cite{TAOCP2}. This can be an obstacle both to
computing addition chain length and proving statements about addition chains.

Despite this, if we define an \emph{addition chain defect}, analogous to
integer complexity defects, we find that they act quite similarly.

As mentioned above, the set of all integer complexity defects
is a well-ordered subset of the real numbers, with order type $\omega^\omega$.
If we define
\[\dft^{\acl}(n):=\acl(n)-\log_2 n,\]
then it was shown in \cite{adcwo} that this is also true for addition chain
defects:

\begin{thm}[Addition chain well-ordering theorem]
\label{adcwothm}
The set \[\D^\acl := \{ \dft^{\acl}(n) : n \in \mathbb{N} \},\] considered as a
subset of the real numbers, is well-ordered and has order type $\omega^\omega$.
\end{thm}

Moreover, as also shown in \cite{adcwo}, stabilization has its analogue as
well:
\begin{thm}
\label{adcstab}
For any natural number $n$, there exists $K\ge 0$ such that, for any $k\ge K$,
\[ \acl(2^k n)=(k-K)+\acl(2^K n). \]
\end{thm}

So we can then ask if the results of this paper will translate to addition
chains.  In \cite{adcwo} it was conjectured:
\begin{conj}
\label{adcconjk}
For each whole number $k$, $\mathscr{D}^\acl\cap[0,k]$ has order type
$\omega^k$.
\end{conj}

We could, then, ask if an even stronger statement might hold, as per this paper:
\begin{conj}
\label{adcconj}
Given $1\le \alpha<\omega^\omega$ an ordinal and $k$ a whole number,
\[\overline{\mathscr{D}^\acl}(\omega^k \alpha)={\mathscr{D}^\acl}(\alpha)+k.\]
Furthermore, 
\[ \overline{\mathscr{D}^\acl_\st} = \mathscr{D}^\acl_\st + \Nn. \]
\end{conj}

We could also ask the same for restricted types of addition chains, such as star
chains or Hansen chains.  Actually,
Theorems~\ref{adcwothm} and \ref{adcstab} were proven for these types of chain
as well in \cite{adcwo}, and
for other sorts obeying quite general conditions, but it is not clear what sort
of conditions would be needed to achieve stronger results such as
Conjecture~\ref{adcconjk} or \ref{adcconj}.

Note that there is no equivalent of the modulo-$3$ results for addition chains,
due to our basic inequality being $\acl(2n)\le \acl(n)+1$, rather than
$\cpx{3n}\le\cpx{n}+3$; see Section~1E of \cite{intdft} for a more detailed
discussion of this point.

\end{document}